\documentclass[a4paper,american,times,3p]{elsarticle}
\usepackage[T1]{fontenc}
\usepackage[utf8]{inputenc}
\usepackage{babel}

\usepackage{amsmath}
\usepackage{amsthm}
\usepackage{amssymb}
\usepackage{graphicx}
\usepackage{color}
\usepackage{bm}
\usepackage{stmaryrd}
\usepackage{stackrel}

\makeatletter

\definecolor{note_fontcolor}{rgb}{1, 0.667969, 0}
\newenvironment{lyxgreyedout}
  {\textcolor{note_fontcolor}\bgroup\ignorespaces}
  {\ignorespacesafterend\egroup}

\usepackage{environ}
\usepackage[colorinlistoftodos,prependcaption,textsize=small]{todonotes}
\RenewEnviron{lyxgreyedout}{\todo[inline]{\protect\BODY}}
\usepackage{etoolbox}

  \theoremstyle{definition}
  \newtheorem{defn}{\protect\definitionname}
  \providecommand{\definitionname}{Definition}
  \theoremstyle{plain}
  \newtheorem{lem}{\protect\lemmaname}
  \providecommand{\lemmaname}{Lemma}
  \theoremstyle{remark}
  \newtheorem*{rem*}{\protect\remarkname}
  \providecommand{\remarkname}{Remark}
  \theoremstyle{plain}
  \newtheorem{thm}{\protect\theoremname}
  \providecommand{\theoremname}{Theorem}

\renewcommand{\vec}[1]{\boldsymbol{#1}}
\usepackage{hyperref}

\makeatother

\newcommand{\e}{\text{e}}
\newcommand{\dd}{\text{d}}
\newcommand{\N}{\mathbb{N}}
\newcommand{\R}{\mathbb{R}}
\newcommand{\Cs}{\text{C}}
\newcommand{\Ls}{\text{L}}
\newcommand{\Hs}{\text{H}}

\newcommand{\D}{\text{\ensuremath{\mathcal{D}}}}
\newcommand{\E}{\text{\ensuremath{\mathcal{E}}}}

\newcommand{\J}{\text{\ensuremath{\mathcal{J}}}}
\newcommand{\Q}{\text{\ensuremath{\mathcal{Q}}}}

\renewcommand{\H}{\text{\ensuremath{\mathcal{H}}}}

\renewcommand{\P}{\text{\ensuremath{\mathcal{P}}}}
\renewcommand{\S}{\text{\ensuremath{\mathcal{S}}}}

\begin{document}
\begin{frontmatter}

\title{Convergence of the Finite Volume Method on Unstructured Meshes for a 3D Phase Field Model of Solidification}

%\titlerunning{Convergence of FV on Unstructured Meshes for a 3D PF Model of Solidification}

\author[fnspe]{Ale\v{s} Wodecki}\ead{ales.wodecki@fjfi.cvut.cz}
\author[fnspe]{Pavel Strachota\corref{cor1}}\ead{pavel.strachota@fjfi.cvut.cz}
\author[fnspe]{Michal Bene\v{s}}\ead{michal.benes@fjfi.cvut.cz}

\cortext[cor1]{Corresponding author. Phone: +420 224 358 563}

\address[fnspe]{Deptartment of Mathematics, Faculty of Nuclear Sciences
and Physical Engineering, Czech Technical University in Prague. Trojanova
13, 120 00 Praha 2, Czech Republic}

\begin{abstract}

We present a convergence result for the finite volume method applied
to a particular phase field problem suitable for simulation of pure substance
solidification. The model consists of the heat equation and the phase field
equation with a general form of the reaction term which encompasses a variety
of existing models governing dendrite growth and elementary interface tracking
problems. We apply the well known compact embedding techniques in the context
of the finite volume method on admissible unstructured polyhedral meshes.
We develop the necessary interpolation theory and derive an a priori estimate
to obtain boundedness of the key terms. Based on this estimate, we conclude
the convergence of all of the terms in the equation system.

\end{abstract}

\begin{keyword}
A priori estimate \sep compact embedding \sep convergence \sep finite volume method \sep phase field problem
\end{keyword}

%%%%% AMS/MSC/PACs/Keywords %%%%%%%%%%%
%% AMS 2020 Subject Classification codes legend:
%% https://mathscinet.ams.org/msc/msc2020.html
%% ------------------------------------------------------------------------
%% 35K51 Partial differential equations | Parabolic equations and systems |
%%       Initial-boundary value problems for second-order parabolic systems
%% 35K57 Partial differential equations | Parabolic equations and systems |
%%       Reaction-diffusion equations
%% 65M20 Numerical analysis | Partial differential equations, initial value
%%       and time-dependent initial-boundary value problems | Method of lines
%%
%% 80A22 Classical thermodynamics, heat transfer | Stefan problems, phase changes, etc.
%% 80M12 Classical thermodynamics, heat transfer | Finite volume methods applied to problems in thermodynamics and heat transfer
%% ------------------------------------------------------------------------

\end{frontmatter}

%% --------------------- THE PAPER TEXT ITSELF ---------------------
%LATEX_EXPORT_BEGIN

\setuptodonotes{disable}

\section{Introduction}

Phase field modeling \cite{Boettinger-SolidificationMicro_Overview,Karma_Rappel-Quant_phase_field_modeling}
and level set methods \cite{Sethian-Level-set-methods} have been
utilized to solve various physical problems that involve moving interfaces.
The phase field model in particular has been deployed to model problems
such as crack propagation \cite{Aranson-Continuum-Field-description},
viscous fingering \cite{Folch-Phase-Field-viscosity}, two phase flow
of immiscible fluids \cite{Luo-TwoPhase_Flow_Phase_Field,Balcazar-Level_set-Bubbles},
phase transitions in porous media \cite{ZakBenesTissa-FreezingInPorousMedium,Alpak-Phase-Field-TwoPhase-Porous,Amiri-Level_set-vs-Phase_Field-in-TwoPhase_Flow},
and prominently, crystal growth. Historically, phase field modelling
was made possible through the study of functionals describing interfacial
energy \cite{Physics-first-order-phase,free-energy-cahn-hil}. Throughout
the years, the model derived using the groundwork laid out by Cahn
and Hilliard gained more concrete form suitable for particular applications
\cite{Allen-Cahn-orig,Karma_Rappel-Num_sim_3D_dendrites,Karma_Rappel-Quant_phase_field_modeling,Physics-first-order-phase}.
More recently, the functional theory of phase transition has been
formalized further leading to a more rigorous treatment \cite{PF-cystal-growth-2009,PF-desity-theory-freezing}
and many new applications were found \cite{PF-multi-comp-flow,PF-fracture-crack-propagation,PF-Asy-Analysis}.
Our main focus is on a phase field model formulation that finds use
in the modeling of dendritic growth and grain evolution during solidification
\cite{Jeong_Goldenfeld_Dantzig-PF-FEM-3D-Flow,Karma_Rappel-Num_sim_3D_dendrites,Karma_Rappel-Quant_phase_field_modeling,Suwa-3D-MPF_grain_growth,dissertation,ISPMA14-Pavel_Ales-ActaPhPoloA}.

In this paper, we present the mathematical analysis of the finite
volume method (FVM) on an unstructured mesh (defined in \cite{FVM-Eymard})
for a system of equations known as the phase field model (PFM) with
a single order parameter. The two predominant classes of results on
the numerical analysis of the PFM are adaptations of \cite{FVM-Eymard}
and the use of approximative sequences of functions and compact embedding
that show existence of the weak solution and convergence at the same
time \cite{Coudiere-FVM-convergence,Benes-Asymptotics,Theoretical-Num-Analysis}.
One can find examples of the first approach in our previous work \cite{BIT-Error_estimate_FVM}
or \cite{PF-dynamic-bc} and the use of compact embedding techniques
is demonstrated in \cite{FVM-PF-VoidElectro,PF-memory,PF-Asy-Analysis,PF-NS-analysis-fluid-interaction}.
In our previous work \cite{BIT-Error_estimate_FVM}, we took the first
mentioned approach and derived estimates that show convergence rates
of the FVM applied to a simplified phase field (Allen-Cahn) equation
on an unstructured mesh . In this paper, we have chosen the second
of the two mentioned approaches. The novelty of our proof lies in
the application of the well known compact embedding techniques \cite{Kufner-Function-Spaces,Theoretical-Num-Analysis}
to the general setting of FVM on an unstructured mesh applied to the
PFM, which gives both the existence of the weak solution and the convergence
of the numerical scheme to this solution. We also introduce an interpolation
theory tailored for this purpose. The weak convergence is given by
an a-priori estimate that is derived in detail.

For the purposes of the analysis, we consider the isotropic PFM with
a generically formulated reaction term, which is well suited for practical
applications: In our related work \cite{arXiv-PhaseField-Focusing-Latent-Heat},
we propose a novel form of the PFM which is compatible with the presented
numerical analysis. In addition, we also introduce its anisotropic
variant. We demonstrate that numerical simulations of rapid solidification
\cite{Bragard_Karma-Higly_undercooled_solidif,Willnecker_Herlach-rapid_solidification_nonequilibrium,Herlach-Non-equilibrium-solidification-metals}
based on this model produce results both in qualitative and quantitative
agreement with experiments. Some traditional models such as \cite{Kobayashi-PF-Dendritic}
can also be treated by the presented framework.

\section{\label{chap:Mathematical-Analysis-of}Problem Formulation}

Let $\Omega\subset\mathbb{R}^{3}$ be a bounded polyhedral domain
and let $\mathcal{J=}(0,T)$ be a time interval. The problem in question
reads

\begin{align}
\frac{\partial u}{\partial t} & =\Delta u+L\frac{\partial p}{\partial t} & \mathcal{\text{in }J}\times\Omega,\label{eq:PhaseIsoHeat-1}\\
\alpha\xi^{2}\frac{\partial p}{\partial t} & =\xi^{2}\Delta p+f(u,p,\xi) & \mathcal{\text{in }J}\times\Omega,\label{eq:PhaseIso-1}\\
u\mid_{t=0} & =u_{\text{ini}},p\mid_{t=0}=p_{\text{ini}} & \mathcal{\text{in }}\Omega,\label{eq:initIso-1}\\
u\mid_{\partial\Omega} & =u_{\partial\Omega} & \mathcal{\text{on }J}\times\partial\Omega,\label{eq:DorN1-1}\\
p\mid_{\partial\Omega} & =p_{\partial\Omega} & \mathcal{\text{on }J}\times\partial\Omega,\label{eq:DorN2-1}
\end{align}
where $u$ and $p$ are the temperature field and the phase field
(i.e. order parameter), respectively. Consider the function $f(u,p,\xi)$
of the form

\begin{equation}
f(u,p,\xi)=f_{0}\left(p\right)+b\beta\xi\varLambda\left(g\left(u,p\right)\right),\label{eq:reaction-term}
\end{equation}
where

\[
f_{0}\left(p\right)\equiv p\left(1-p\right)\left(p-\frac{1}{2}\right),
\]
$L,\alpha,\beta,b$ are positive constants \cite{Benes-Math_comp_aspects_solid}
and $\xi$ is a parameter associated with the thickness of the diffuse
interface. Assume that the initial conditions satisfy $u_{\text{ini}},p_{\text{ini}}\in\Cs^{2}\left(\Omega\right)$
and the Dirichlet boundary conditions satisfy $u_{\partial\Omega},p_{\partial\Omega}\in\Cs\left(\partial\Omega\right)$.%
\begin{lyxgreyedout}
\textbf{NOTE:} Originally, we had $u_{\text{ini}},p_{\text{ini}}\in\Cs^{\infty}\left(\Omega\right)$
here. The assumption $u_{\text{ini}},p_{\text{ini}}\in\Cs^{2}\left(\Omega\right)$
is explicitly stated again in the lemmas and theorems that use it.
$u_{\partial\Omega},p_{\partial\Omega}\in\Cs\left(\partial\Omega\right)$
is later replaced by $u_{\partial\Omega}=p_{\partial\Omega}=0$. However,
it is probably reasonable to have such assumptions at the beginning
since the finite volume scheme doesn't make practical sense without
them.%
\end{lyxgreyedout}

The function $g$ in the reaction term (\ref{eq:reaction-term}) can
be an arbitrary function of $u$ and $p$, which covers several well
known variations of the phase field model such as the Kobayashi model
\cite{Kobayashi-Anis_curvature_crystals} or some of the simpler models
proposed in \cite{Benes-Math_comp_aspects_solid,Benes-Anisotropic_phase_field_model}.
In addition, we require that $g$ is subject to a ``limiter'' function
$\Lambda\in\Cs^{1}\left(\R\right)$ that bounds the range of $g$
to a fixed interval $\left(H_{\text{inf}},H_{\text{sup}}\right)$,
which is vital to the convergence analysis. For example, a suitable
choice is

\[
\varLambda\left(x\right)=\begin{cases}
x & x\in\left[H_{0},H_{1}\right],\\
H_{1}+\frac{2}{\pi}\left(H_{\text{sup}}-H_{1}\right)\text{arctan}\left(\frac{\pi}{2}\frac{x-H_{1}}{H_{\text{sup}}-H_{1}}\right) & x\in\left(H_{1},+\infty\right),\\
H_{0}+\frac{2}{\pi}\left(H_{\text{inf}}-H_{0}\right)\text{arctan}\left(\frac{\pi}{2}\frac{x-H_{0}}{H_{\text{inf}}-H_{0}}\right) & x\in\left(-\infty,H_{0}\right),
\end{cases}
\]
where $H_{\text{inf}}<H_{0}<H_{1}<H_{\text{sup}}$. On the other hand,
during simulations using the above cited models, the term $g\left(u,p\right)$
remains bounded and thus introducing $\Lambda$ with a sufficiently
wide interval $\left[H_{0},H_{1}\right]$ (where $\Lambda\left(x\right)=x$)
has no practical implications.

\section{Finite Volume Method on an Unstructured Mesh\label{sec:Finite-volume-method}}

Let $B\subset\mathbb{R}^{3}$. The 3- or 2-dimensional Lebesgue measure
of the set $B$ will be denoted $m(B)$ or $\widetilde{m}(B)$, respectively.
In cases where the dimension of the object in question is clear, the
tilde will be omitted. The definitions in this section agree with
\cite{FVM-Eymard}.
\begin{defn}
\label{def:AdmissibleMesh}Let $\Pi\subset2^{\Omega}$ be a set such
that for all $K\in\Pi$, $K$ is polygonal and convex, and let $\E$
be the set of all faces. If $\Pi$ and $\E$ satisfy
\end{defn}
\begin{enumerate}
\item $\overline{\underset{K\in\varPi}{\uplus}K}=\Omega$
\item $\forall K\in\Pi:\exists\E_{K}\subset\E:\partial K=\underset{\sigma\in\E_{K}}{\cup}\overline{\sigma}$
\item $\left(\forall K,L\in\Pi\right)\left(K\neq L\Rightarrow\left(\widetilde{m}\left(\overline{K}\cap\overline{L}\right)=0\vee\left(\exists\sigma\in\E\right)\left(\overline{K}\cap\overline{L}=\sigma\right)\right)\right)$
\item $\left(\exists P\subset\Omega\right)\left(\left(\forall K\in\Pi\right)\left(\exists_{1}\vec{x}_{K}\in P\right)\left(\vec{x}_{K}\in\overline{K}\right)\wedge\left(\forall\vec{x}\in P\right)\left(\exists K\in\Pi\right)\left(\vec{x}\in\overline{K}\right)\right)$
\item \label{enu:admissibility-orthogonality}$\left(\forall K,L\in\Pi\right)\left(K\neq L\Rightarrow\left(\vec{x}_{K}\neq\vec{x}_{L}\wedge\left(\overline{\vec{x}_{K}\vec{x}_{L}}\perp\overline{K}\cap\overline{L}\right)\right)\right),$
\end{enumerate}
then we call $\Pi$ an admissible mesh and $\E$ the set of faces
associated with the mesh $\Pi$.

The elements of $\Pi$ are called finite volumes (or cells). Definition
\ref{def:AdmissibleMesh} ensures that two finite volumes only intersect
at a point or a face and that the line segment connecting two significant
points representing two adjacent volumes will always be perpendicular
to their common face. Let $\D_{K,L}$ denote the line segment connecting
two significant points $\overline{\vec{x}_{K}\vec{x}_{L}}\equiv\D_{K,L}$.
Let $K\in\Pi$ such that $\exists\sigma\in\E_{K}:\sigma\subset\partial\Omega$.
Assume that $\vec{x}_{K}\notin\sigma$, then denote $\D_{K,\sigma}\equiv\overline{\vec{x}_{K}\vec{y}_{\sigma}}$,
where $\vec{y}_{\sigma}$ is chosen such that $\D_{K,\sigma}\perp\sigma$.
In agreement with \cite{FVM-Eymard}, we use the following conventions:
\begin{itemize}
\item $\E_{\text{ext}}=\left\{ \sigma\in\E:\sigma\subset\partial\Omega\right\} ,\E_{\text{int}}=\E-\E_{\text{ext}}$.
\item For a face such that $\sigma=K\left|L\right.\equiv\overline{K}\cap\overline{L}$,
let $d_{\sigma}=\left|\D_{K,L}\right|$ be the Euclidean distance
between points $\vec{x}_{K}$ and $\vec{x}_{L}$. Similarly, we define
$d_{K,\sigma}=\left|\D_{K,\sigma}\right|$. For $\sigma\in\E_{K}\cap\E_{\text{ext}}$,
we put $d_{\sigma}\equiv d_{K,\sigma}.$
\item $\tau_{\sigma}\equiv\frac{m\left(\sigma\right)}{d_{\sigma}}.$
\item $\H^{\Pi}$ denotes the set of all functions $w:P\rightarrow\mathbb{R}$,
where $P$ has the meaning of the set of all significant points of
$\Pi$, given by Definition \ref{def:AdmissibleMesh}. For $w\in\H^{\Pi}$,
the simplified notation 
\begin{equation}
w_{K}\equiv w\left(\vec{x}_{K}\right)\qquad\forall K\in\Pi\label{eq:wK-simplified-notation}
\end{equation}
will be used.
\end{itemize}
\begin{rem*}
Every admissible mesh satisfies
\begin{equation}
\underset{\sigma\in\E}{\sum}m\left(\sigma\right)d_{\sigma}=3m\left(\Omega\right).\label{eq:dual_cell_volume_sum_3D}
\end{equation}
\end{rem*}
\begin{proof}
For each $\sigma\in\E_{\text{int}}$, $\sigma=K|L$, we have $d_{\sigma}=d_{K,\sigma}+d_{L,\sigma}$.
Thanks to the orthogonality condition \ref{enu:admissibility-orthogonality}
in Definition \ref{def:AdmissibleMesh}, the expression $\underset{\sigma\in\E}{\sum}\frac{1}{3}m\left(\sigma\right)d_{\sigma}$
represents the sum of volumes of the pyramids with base $\sigma$
and some vertex $\vec{x}_{K}$ such that $\sigma\in\E_{K}$. These
pyramids exactly cover the volume of all cells in $\Pi$ (recall the
definition of $d_{\sigma}$ when $\sigma\in\E_{\text{int }}$ and
when $\sigma\in\E_{\text{ext}}$), which in turn cover the whole domain
$\Omega$.
\end{proof}
Let $u_{\Pi},p_{\Pi}:\J\to\H^{\Pi}$ be the numerical solutions of
problem (\ref{eq:PhaseIsoHeat-1})--(\ref{eq:DorN2-1}) and similarly
to (\ref{eq:wK-simplified-notation}), denote 
\[
u_{K}\left(t\right)\equiv u_{\Pi}\left(t,\vec{x}_{K}\right),\text{ }p_{K}\left(t\right)\equiv p_{\Pi}\left(t,\vec{x}_{K}\right).
\]
 Consider the following reformulation of equations (\ref{eq:PhaseIsoHeat-1}),
(\ref{eq:PhaseIso-1})

\begin{align}
\frac{\partial u}{\partial t}-\Delta u & =L\frac{\partial p}{\partial t},\label{eq:HeatToInteg}\\
\alpha\frac{\partial p}{\partial t}-\Delta p & =\frac{1}{\xi^{2}}f_{0}\left(p\right)-b\beta\frac{1}{\xi}\varLambda\left(g\left(u,p\right)\right).\label{eq:phaseToInteg}
\end{align}
By integrating (\ref{eq:HeatToInteg}) and (\ref{eq:phaseToInteg})
over a finite volume $K\in\Pi$ and using Green's formula, we obtain

\begin{align}
\int_{K}\frac{\partial u}{\partial t}\left(t,\vec{x}\right)\dd\vec{x} & -\int_{\partial K}\nabla u\left(t,\vec{x}\right)\cdot\vec{n}\dd S=L\int_{K}\frac{\partial p}{\partial t}\left(t,\vec{x}\right)\dd\vec{x},\label{eq:heatAlmost}
\end{align}
\begin{align}
 & \alpha\int_{K}\frac{\partial p}{\partial t}\left(t,\vec{x}\right)\dd\vec{x}-\int_{\partial K}\nabla p\left(t,\vec{x}\right)\cdot\vec{n}\dd S\label{eq:phaseAlmost}\\
= & \frac{1}{\xi^{2}}\int_{K}f_{0}\left(p\left(t,\vec{x}\right)\right)\dd\vec{x}-\frac{b\beta}{\xi}\int_{K}\varLambda\left(g\left(u\left(\vec{x},t\right),p\left(\vec{x},t\right)\right)\right)\dd\vec{x},\nonumber 
\end{align}
where $\vec{n}$ is the outward pointing unit normal vector to $\partial K$.
Using additivity of the surface integral, we can rewrite (\ref{eq:heatAlmost})
and (\ref{eq:phaseAlmost}) as

\begin{align}
\int_{K}\frac{\partial u}{\partial t}\left(t,\vec{x}\right)\dd\vec{x}+\underset{\sigma\in\E_{K}}{\sum}-\int\nabla u\left(t,\vec{x}\right)\cdot\vec{n}_{K,\sigma}\dd S & =L\int_{K}\frac{\partial p}{\partial t}\left(t,\vec{x}\right)\dd\vec{x},\label{eq:heatAlmost-1}
\end{align}
\begin{align}
 & \alpha\int_{K}\frac{\partial p}{\partial t}\left(t,\vec{x}\right)\dd\vec{x}+\underset{\sigma\in\E_{K}}{\sum}-\int_{\sigma}\nabla p\left(t,\vec{x}\right)\cdot\vec{n}_{K,\sigma}\dd S\label{eq:phaseAlmost-1}\\
= & \frac{1}{\xi^{2}}\int_{K}f_{0}\left(p\left(t,\vec{x}\right)\right)\dd\vec{x}-\frac{b\beta}{\xi}\int_{K}\varLambda\left(g\left(u\left(\vec{x},t\right),p\left(\vec{x},t\right)\right)\right)\dd\vec{x},\nonumber 
\end{align}
where $\vec{n}_{K,\sigma}$ is the unit normal vector to $\sigma\in\E_{K}$
pointing out of $\partial K$. The following approximations may be
applied to the individual terms in equations (\ref{eq:heatAlmost-1})
and (\ref{eq:phaseAlmost-1})

\begin{align}
\int_{K}\frac{\partial p}{\partial t}\left(t,\vec{x}\right)\dd\vec{x} & \rightarrow m\left(K\right)\dot{p}_{K}\left(t\right),\label{eq:aprox1}\\
-\int_{\sigma}\nabla p\left(t,\vec{x}\right)\cdot\vec{n}_{K,\sigma}\dd S & \rightarrow F_{K,\sigma}\left(p_{\Pi}\left(t\right),p_{\partial\Omega}\left(t\right)\right),\label{eq:aprox2}\\
\int_{K}f_{0}\left(p\left(t,\vec{x}\right)\right)\dd\vec{x} & \rightarrow f_{0,K}\left(t\right)=m\left(K\right)f_{0}\left(p_{K}\left(t\right)\right),\label{eq:aprox3}\\
\int_{K}\frac{\partial u}{\partial t}\left(t,\vec{x}\right)\dd\vec{x} & \rightarrow m\left(K\right)\dot{u}_{K}\left(t\right),\label{eq:aprox4}\\
-\int_{\sigma}\nabla u\left(t,\vec{x}\right)\cdot\vec{n}_{K,\sigma}\dd S & \rightarrow F_{K,\sigma}\left(u_{\Pi}\left(t\right),u_{\partial\Omega}\left(t\right)\right),\label{eq:aprox5}\\
\int_{K}\varLambda\left(g\left(u\left(\vec{x},t\right),p\left(\vec{x},t\right)\right)\right)\dd\vec{x} & \rightarrow\varLambda_{K}\left(t\right)=m\left(K\right)\varLambda\left(g\left(u_{K}\left(t\right),p_{K}\left(t\right)\right)\right).\label{eq:aprox6}
\end{align}
For $w\in\H_{h}$ and $w_{\partial\Omega}\in\Cs\left(\partial\Omega\right)$,
we define
\begin{align}
F_{K,\sigma}\left(w,w_{\partial\Omega}\right) & =\begin{cases}
-\tau_{\sigma}\left(w_{L}-w_{K}\right) & \forall\sigma\in\E_{\text{int}},\sigma=K|L,\\
-\tau_{\sigma}\left(w_{\partial\Omega}\left(y_{\sigma}\right)-w_{K}\right)\text{ } & \forall\sigma\in\E_{\text{ext}}\cap\E_{K},
\end{cases}\nonumber \\
F_{K}\left(w,w_{\partial\Omega}\right) & =\sum_{\sigma\in\mathcal{E}_{K}}F_{K,\sigma}\left(w,w_{\partial\Omega}\right)\label{eq:flux-term}
\end{align}
for each $K\in\Pi$. Finally, in the sense of (\ref{eq:wK-simplified-notation}),
it is natural to define a function $F\left(w,w_{\partial\Omega}\right)\in\H^{\Pi}$
as
\begin{equation}
F\left(w,w_{\partial\Omega}\right)\left(\vec{x}_{K}\right)=\frac{1}{m\left(K\right)}F_{K}\left(w,w_{\partial\Omega}\right).\label{eq:flux-mesh-function}
\end{equation}

The approximations (\ref{eq:aprox1})--(\ref{eq:aprox5}) give rise
to the semi discrete scheme for the finite volume method

\begin{align}
m\left(K\right)\dot{u}_{K}\left(t\right)+F_{K}\left(u_{\Pi}\left(t\right),u_{\partial\Omega}\left(t\right)\right) & =Lm\left(K\right)\dot{p}_{K}\left(t\right),\label{eq:semiDisSchemeHeat}\\
\alpha m\left(K\right)\dot{p}_{K}\left(t\right)+F_{K}\left(p_{\Pi}\left(t\right),p_{\partial\Omega}\left(t\right)\right) & =\frac{1}{\xi^{2}}f_{0,K}\left(t\right)\label{eq:semiDisSchemePhase}\\
 & -\frac{b\beta}{\xi}\varLambda_{K}\left(t\right)m\left(K\right),\nonumber 
\end{align}
$\forall K\in\Pi$, $\forall t\in\mathcal{J}$, with the initial conditions

\begin{align}
p_{K}\left(0\right) & =p\left(0,\vec{x}_{K}\right),\label{eq:semiDisSchemeInitU}\\
u_{K}\left(0\right) & =u\left(0,\vec{x}_{K}\right).\text{ }\text{ }\text{ }\text{ }\forall K\in\Pi\label{eq:semiDisSchemeInitP}
\end{align}

The procedure of discretization using FVM led to the system of ordinary
differential equations (\ref{eq:semiDisSchemeHeat}) and (\ref{eq:semiDisSchemePhase}).
These equations will be used to prove the existence, uniqueness of
the weak solution, and convergence of the numerical scheme.

\section{Interpolation Theory}

In addition to the approximations presented in the previous section
\cite{FVM-Eymard}, we require a suitable interpolation theory that
allows us to perform the proof. For the rest of this section, assume
that $\Pi$ is an admissible mesh in the sense of Definition \ref{def:AdmissibleMesh}.
The notion of a dual mesh is crucial for defining a piecewise linear
interpolation of a function from $\H^{\Pi}.$
\begin{defn}
For all $\vec{x}_{K},\vec{x}_{L}\in P$ such that $\vec{x}_{K}\neq\vec{x}_{L}$,
$\sigma\equiv\overline{K}\cap\overline{L}$, $\left(\overline{K}\cap\overline{L}\right)\neq0$,
define $\varSigma_{K,L}\equiv\text{convhull}\left\{ \vec{x}_{K},\vec{x}_{L},\sigma\right\} $.
For all $\sigma\in\E$, define $\Sigma_{K,\sigma}\equiv\text{convhull}\left\{ \vec{x}_{K},\sigma\right\} $,
where $\vec{x}_{K}$ is the significant point of the control volume
for which $\sigma\in\E_{K}$. Then the dual mesh $\Pi^{*}$ of $\Pi$
is defined as 

\[
\Pi^{*}=\left\{ \Sigma_{K,L}\left|K\left|L\right.\in\E_{\text{int}}\right.\right\} \cup\left\{ \Sigma_{K,\sigma}\left|\sigma\in\E_{\text{ext}}\right.\right\} .
\]
\end{defn}
The elements of the dual mesh are depicted in Figure \ref{fig:Elements-of-the-DualMesh}.
Next, two interpolation operators will be introduced.

\begin{figure}
\begin{centering}
\includegraphics[width=100mm]{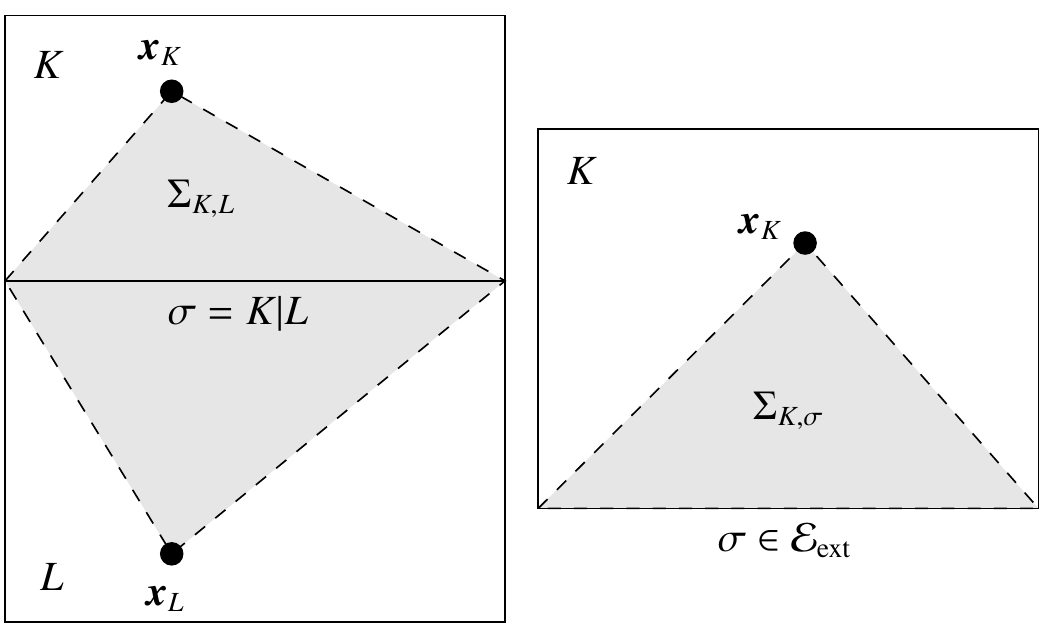}
\par\end{centering}
\caption{\label{fig:Elements-of-the-DualMesh}Construction of the dual cell
$\Sigma_{K,L}$ between the cells $K,L$ and the dual cell $\Sigma_{K,\sigma}$
at the boundary of $\Omega$.}
\end{figure}

\begin{defn}
\label{def:InterpolationOperators}Let $w\in\H^{\Pi}$. We define
the piecewise constant interpolation operator $\mathcal{S}_{\Pi}:\H^{\Pi}\rightarrow\left\{ f\left|f:\Omega\rightarrow\mathbb{R}\right.\right\} $
as

\[
\left(\mathcal{S}_{\Pi}w\right)\left(\vec{x}\right)\equiv w\left(\vec{x}_{K}\right),\forall K\in\Pi,\forall\vec{x}\in K.
\]
The piecewise linear interpolation $\mathcal{Q}_{\Pi}:\H^{\Pi}\rightarrow\left\{ f\left|f:\Omega\rightarrow\mathbb{R}\right.\right\} $
is defined as 

\[
\left(\Q_{\Pi}w\right)\left(\vec{x}\right)\equiv w\left(\vec{x}_{K}\right)+\frac{1}{\left\Vert \vec{x}_{L}-\vec{x}_{K}\right\Vert ^{2}}\left(\vec{x}-\vec{x}_{K}\right)\cdot\left(\vec{x}_{L}-\vec{x}_{K}\right)\left[w\left(\vec{x}_{L}\right)-w\left(\vec{x}_{K}\right)\right]
\]

\[
\forall\varSigma_{K,L}\in\Pi^{*},\vec{x}\in\varSigma_{K,L},
\]
and
\[
\left(\Q_{\Pi}w\right)\left(\vec{x}\right)\equiv w\left(\vec{x}_{K}\right)+\frac{1}{\left\Vert \vec{y}_{\sigma}-\vec{x}_{K}\right\Vert ^{2}}\left(\vec{x}-\vec{x}_{K}\right)\cdot\left(\vec{y}_{\sigma}-\vec{x}_{K}\right)\left[-w\left(\vec{x}_{K}\right)\right]
\]
\[
\forall\varSigma_{K,\sigma}\in\Pi^{*},\vec{x}\in\varSigma_{K,\sigma}.
\]
\end{defn}
\begin{defn}
\label{def:inner-product-on-mesh}For $v,w\in\H^{\Pi}$, we define
\end{defn}
\begin{align*}
\left(v,w\right)_{\Pi} & \equiv\underset{K\in\Pi}{\sum}m\left(K\right)v_{K}w_{K},\\
\left\Vert w\right\Vert _{\Pi} & \equiv\sqrt{\left(w,w\right)_{\Pi}^{2}},\\
\left\llbracket w\right\rrbracket _{\Pi}^{2} & \equiv\underset{{\sigma\in\E_{K}\cap\E_{\text{int}}\atop \sigma=K|L}}{\sum}\tau_{\sigma}\left(w_{K}-w_{L}\right)^{2}+\underset{{\sigma\in\E_{\text{ext}}\atop \sigma\in\mathcal{E}_{K}}}{\sum}\tau_{\sigma}w_{K}^{2}.
\end{align*}

\begin{lem}
\label{lem:discrete-to-L2-relationships}Let $v,w\in\H^{\Pi}$, then
the relationships
\begin{align}
\left(v,w\right)_{\Pi} & =\left(\mathcal{S}_{\Pi}v,\mathcal{S}_{\Pi}w\right)_{\Ls^{2}\left(\Omega\right)},\label{eq:innerp1}\\
\left\Vert w\right\Vert _{\Pi}^{2} & =\left\Vert \mathcal{S}_{\Pi}w\right\Vert _{\Ls^{2}\left(\Omega\right)}^{2},\label{eq:norm2}\\
\left\llbracket w\right\rrbracket _{\Pi}^{2} & =3\left\Vert \left|\nabla\mathcal{Q}_{\Pi}w\right|\right\Vert _{\Ls^{2}\left(\Omega\right)}^{2}.\label{eq:norm3}
\end{align}
hold.
\end{lem}
\begin{proof}
(\ref{eq:innerp1}) and (\ref{eq:norm2}) follow directly from the
definition of the respective inner products. To prove (\ref{eq:norm3}),
we recall Definition (\ref{def:InterpolationOperators}), relation
(\ref{eq:dual_cell_volume_sum_3D}), and the notations of Figure \ref{fig:dual-cell}
so that we can write
\begin{align*}
\left\llbracket w\right\rrbracket _{\Pi}^{2} & =\underset{{\sigma\in\E_{K}\cap\E_{\text{int}}\atop \sigma=K|L}}{\sum}\tau_{\sigma}\left(w_{K}-w_{L}\right)^{2}+\underset{{\sigma\in\E_{\text{ext}}\atop \sigma\in\mathcal{E}_{K}}}{\sum}\tau_{\sigma}w_{K}^{2}\\
 & =\underset{{\sigma\in\E_{K}\cap\E_{\text{int}}\atop \sigma=K|L}}{\sum}\frac{m\left(\sigma\right)}{d_{\sigma}}\left(w_{K}-w_{L}\right)^{2}+\underset{{\sigma\in\E_{\text{ext}}\atop \sigma\in\mathcal{E}_{K}}}{\sum}\frac{m\left(\sigma\right)}{d_{\sigma}}w_{K}^{2}\\
 & =\underset{{\sigma\in\E_{K}\cap\E_{\text{int}}\atop \sigma=K|L}}{\sum}m\left(\sigma\right)\left(d_{K,\sigma}+d_{L,\sigma}\right)\left(\frac{w_{K}-w_{L}}{d_{\sigma}}\right)^{2}+\underset{{\sigma\in\E_{\text{ext}}\atop \sigma\in\mathcal{E}_{K}}}{\sum}m\left(\sigma\right)d_{K,\sigma}\left(\frac{w_{K}}{d_{\sigma}}\right)^{2}\\
 & =\underset{{\sigma\in\E_{K}\cap\E_{\text{int}}\atop \sigma=K|L}}{\sum}m\left(\sigma\right)\left(d_{K,\sigma}+d_{L,\sigma}\right)\left|\nabla\mathcal{Q}_{\Pi}w\left(\vec{y}_{\sigma}\right)\right|^{2}\\
 & +\underset{{\sigma\in\E_{\text{ext}}\atop \sigma\in\mathcal{E}_{K}}}{\sum}m\left(\sigma\right)d_{K,\sigma}\left|\nabla\mathcal{Q}_{\Pi}w\left(\vec{y}_{\sigma}\right)\right|^{2}\\
 & =\sum_{K\in\Pi}\sum_{\sigma\in\mathcal{E}_{K}}m\left(\sigma\right)d_{K,\sigma}\left|\nabla\mathcal{Q}_{\Pi}w\left(\vec{y}_{\sigma}\right)\right|^{2}\\
 & =\sum_{K\in\Pi}\sum_{\sigma\in\mathcal{E}_{K}}3\int\limits _{\Sigma_{K,\sigma}}\left|\nabla\mathcal{Q}_{\Pi}w\left(\vec{x}\right)\right|^{2}\dd\vec{x}=3\int\limits _{\Omega}\nabla\mathcal{Q}_{\Pi}w\left(\vec{x}\right)\cdot\nabla\mathcal{Q}_{\Pi}w\left(\vec{x}\right)\dd\vec{x}\\
 & =3\left\Vert \left|\nabla\mathcal{Q}_{\Pi}w\right|\right\Vert _{\Ls^{2}\left(\Omega\right)}^{2}.
\end{align*}
\end{proof}
\begin{lem}
\label{lem:Q-S-estimate}Let $\Pi$ be an admissible mesh. Then there
exists a mesh dependent constant $C_{\Pi}>0$ such that for all $w\in\H^{\Pi}$,
the inequality

\begin{equation}
\left\Vert \Q_{\Pi}w\right\Vert _{\Ls^{2}\left(\Omega\right)}\leq C_{\Pi}\left\Vert \mathcal{S}_{\Pi}w\right\Vert _{\Ls^{2}\left(\Omega\right)}\label{eq:ineqOf2InterpolOps}
\end{equation}
holds.
\end{lem}
\begin{proof}
Let $\Sigma_{K,L}\in\Pi^{*}$ be a cell of the dual mesh and denote
$\sigma=K|L$. From Definition \ref{def:InterpolationOperators},
we easily get
\begin{equation}
\left\Vert \S_{\Pi}w\right\Vert _{\Ls^{2}\left(\varSigma_{K,L}\right)}^{2}=\frac{1}{3}\widetilde{m}\left(\sigma\right)\left(d_{K,\sigma}w_{K}^{2}+d_{L,\sigma}w_{L}^{2}\right).\label{eq:ConstantInterpol}
\end{equation}
To evaluate $\left\Vert \mathcal{Q}_{\Pi}w\right\Vert _{\Ls^{2}\left(\varSigma_{K,L}\right)}^{2}$,
we represent $\Sigma_{K,L}$ (see Figure \ref{fig:dual-cell}
\begin{figure}
\begin{centering}
\includegraphics[width=90mm]{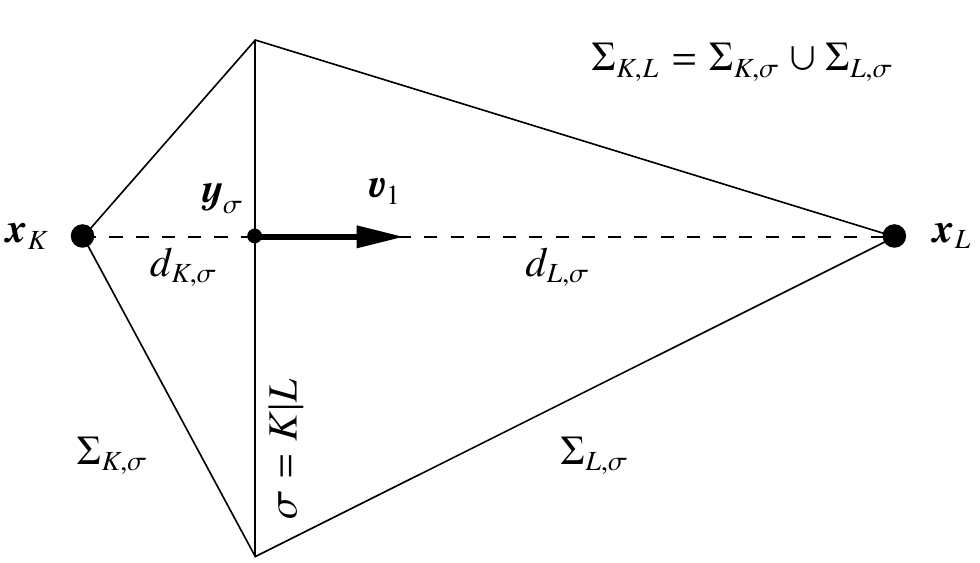}\caption{\label{fig:dual-cell}On the calculation of $\left\Vert \protect\Q_{\Pi}u\right\Vert _{\protect\Ls^{2}\left(\varSigma_{K,L}\right)}^{2}$
on the dual cell $\Sigma_{K,L}$.}
\par\end{centering}
\end{figure}
) in an affine space centered at $\vec{y}_{\sigma}$ with an orthonormal
basis $V=\left(\vec{v}_{1},\vec{v}_{2},\vec{v}_{3}\right)$ where
$\vec{v}_{1}=\frac{\vec{x}_{L}-\vec{x}_{K}}{\left|\vec{x}_{L}-\vec{x}_{K}\right|}$.
The coordinates of a point $\vec{x}\in\Sigma_{K,L}$ in $V$ will
be denoted by $\left(\alpha_{1},\alpha_{2},\alpha_{3}\right)$, i.e.
\[
\vec{x}=\vec{y}_{\sigma}+\sum_{i=1}^{3}\alpha_{i}\vec{v}_{i}.
\]
By means of this transformation, we have
\begin{equation}
\left\Vert \Q_{\Pi}w\right\Vert _{\Ls^{2}\left(\varSigma_{K,L}\right)}^{2}=\int\limits _{-d_{K,\sigma}}^{d_{L,\sigma}}\dd\alpha_{1}\int\limits _{\tilde{S}\left(\alpha_{1}\right)}\dd\left(\alpha_{2},\alpha_{3}\right)\left[\left(\Q_{\Pi}w\right)\left(\vec{y}_{\sigma}+\sum_{i=1}^{3}\alpha_{i}\vec{v}_{i}\right)\right]^{2}\label{eq:linearInterpol}
\end{equation}
where
\[
\tilde{S}\left(\alpha_{1}\right)=\left\{ \left(\alpha_{2},\alpha_{3}\right)\in\mathbb{R}^{2};\vec{y}_{\sigma}+\sum_{i=1}^{3}\alpha_{i}\vec{v}_{i}\in\Sigma_{K,L}\right\} .
\]
In addition, denote by $S\left(\alpha_{1}\right)$ the corresponding
planar cut through $\Sigma_{K,L}$, i.e.
\[
S\left(\alpha_{1}\right)=\left\{ \vec{y}_{\sigma}+\sum_{i=1}^{3}\alpha_{i}\vec{v}_{i};\left(\alpha_{2},\alpha_{3}\right)\in\tilde{S}\left(\alpha_{1}\right)\right\} .
\]
In (\ref{eq:linearInterpol}), Definition \ref{def:InterpolationOperators}
allows to evaluate
\begin{equation}
\left(\Q_{\Pi}w\right)\left(\vec{y}_{\sigma}+\sum_{i=1}^{3}\alpha_{i}\vec{v}_{i}\right)=w_{K}+\frac{\alpha_{1}+d_{K,\sigma}}{d_{L,\sigma}+d_{K,\sigma}}\left(w_{L}-w_{K}\right)\label{eq:Qw-in-dual-cell}
\end{equation}
which only depends on $\alpha_{1}$. Next, the term $\int\limits _{\tilde{S}\left(\alpha_{1}\right)}\dd\left(\alpha_{2},\alpha_{3}\right)$
represents the surface area of $S\left(\alpha_{1}\right)$ which satisfies
\begin{equation}
\int\limits _{\tilde{S}\left(\alpha_{1}\right)}\dd\left(\alpha_{2},\alpha_{3}\right)=\tilde{m}\left(S\left(\alpha_{1}\right)\right)=\begin{cases}
\tilde{m}\left(\sigma\right)\frac{d_{K,\sigma}+\alpha_{1}}{d_{K,\sigma}} & \alpha_{1}\in\left[-d_{K,\sigma},0\right],\\
\tilde{m}\left(\sigma\right)\frac{d_{L,\sigma}-\alpha_{1}}{d_{L,\sigma}} & \alpha_{1}\in\left[0,d_{L,\sigma}\right].
\end{cases}\label{eq:surfaceAreaCut}
\end{equation}
Plugging (\ref{eq:Qw-in-dual-cell}) and (\ref{eq:surfaceAreaCut})
into (\ref{eq:linearInterpol}) and performing the integration%
\begin{lyxgreyedout}
\textbf{EXPLANATION:} See the \texttt{Lemma\_Q-S.mw} MAPLE worksheet%
\end{lyxgreyedout}
{} with respect to $\alpha_{1}$ leads to
\begin{align*}
\left\Vert \Q_{\Pi}w\right\Vert _{\Ls^{2}\left(\varSigma_{K,L}\right)}^{2} & =\frac{1}{12}\frac{\tilde{m}\left(\sigma\right)}{d_{K,\sigma}+d_{L,\sigma}}\bigg[\left(d_{K,\sigma}^{2}+3d_{K,\sigma}d_{L,\sigma}+3d_{L,\sigma}^{2}\right)w_{K}^{2}\\
 & +2\left(d_{K,\sigma}^{2}+3d_{K,\sigma}d_{L,\sigma}+d_{L,\sigma}^{2}\right)w_{K}w_{L}\\
 & +\left(3d_{K,\sigma}^{2}+3d_{K,\sigma}d_{L,\sigma}+d_{L,\sigma}^{2}\right)w_{L}^{2}\bigg].
\end{align*}
By the Young inequality, this can be estimated as
\begin{align*}
\left\Vert \Q_{\Pi}w\right\Vert _{\Ls^{2}\left(\varSigma_{K,L}\right)}^{2} & \leq\frac{1}{6}\frac{\tilde{m}\left(\sigma\right)}{d_{K,\sigma}+d_{L,\sigma}}\bigg[\left(d_{K,\sigma}^{2}+3d_{K,\sigma}d_{L,\sigma}+2d_{L,\sigma}^{2}\right)w_{K}^{2}\\
 & +\left(2d_{K,\sigma}^{2}+3d_{K,\sigma}d_{L,\sigma}+d_{L,\sigma}^{2}\right)w_{L}^{2}\bigg].
\end{align*}
The equality (\ref{eq:ConstantInterpol}) can be rewritten into a
similar form 
\begin{align*}
\left\Vert \mathcal{S}_{\Pi}w\right\Vert _{\Ls^{2}\left(\varSigma_{K,L}\right)}^{2} & =\frac{1}{6}\frac{\tilde{m}\left(\sigma\right)}{d_{K,\sigma}+d_{L,\sigma}}\bigg[\left(2d_{K,\sigma}^{2}+2d_{K,\sigma}d_{L,\sigma}\right)w_{K}^{2}\\
 & +\left(2d_{K,\sigma}d_{L,\sigma}+2d_{L,\sigma}^{2}\right)w_{L}^{2}\bigg].
\end{align*}
Defining
\[
C_{K,L}\equiv\max\left\{ \frac{d_{K,\sigma}^{2}+3d_{K,\sigma}d_{L,\sigma}+2d_{L,\sigma}^{2}}{2d_{K,\sigma}^{2}+2d_{K,\sigma}d_{L,\sigma}},\frac{2d_{K,\sigma}^{2}+3d_{K,\sigma}d_{L,\sigma}+d_{L,\sigma}^{2}}{2d_{K,\sigma}d_{L,\sigma}+2d_{L,\sigma}^{2}}\right\} ,
\]
we observe that
\[
\left\Vert \Q_{\Pi}w\right\Vert _{\Ls^{2}\left(\varSigma_{K,L}\right)}^{2}\leq C_{K,L}\left\Vert \S_{\Pi}w\right\Vert _{\Ls^{2}\left(\varSigma_{K,L}\right)}^{2}.
\]
Note that $C_{K,L}$ is only dependent on the ratio between $d_{K,\sigma}$
and $d_{L,\sigma}$, not on their absolute values. Namely, it is independent
of $w$ and of any mesh refinement as long as the geometry of the
cells remains unchanged%
\begin{lyxgreyedout}
\textbf{TODO:} The note about mesh geometry can possibly be rewritten
or omitted.%
\end{lyxgreyedout}
.

An analogous inequality in the form
\[
\left\Vert \Q_{\Pi}w\right\Vert _{\Ls^{2}\left(\varSigma_{K,\sigma}\right)}^{2}\leq C_{K,\sigma}\left\Vert \S_{\Pi}w\right\Vert _{\Ls^{2}\left(\varSigma_{K,\sigma}\right)}^{2}
\]
can be derived for the dual cells $\Sigma_{K,\sigma}$ at the boundary
of $\Omega$.%
\begin{lyxgreyedout}
\textbf{NOTE:} It is also possible to mirror $\Sigma_{K,\sigma}$
along $\sigma$ to obtain a ``ghost'' cell $L$ with $w_{L}=-w_{K}$
which ensures that $\left(\mathcal{Q}w\right)\left(\vec{y}_{\sigma}\right)=0$.
Then we perform the procedure exactly as shown above and the resulting
norms are just $\frac{1}{2}$ of the obtained values.%
\end{lyxgreyedout}
{} Define
\[
C_{\Pi}\equiv\text{max}\left(\left\{ C_{K,L};K|L\in\E_{\text{int}}\right\} \cup\left\{ C_{K,\sigma};\sigma\in\E_{\text{ext}}\right\} \right).
\]
Summing up the estimates
\[
\left\Vert \Q_{\Pi}w\right\Vert _{\Ls^{2}\left(\varSigma_{K,L}\right)}^{2}\leq C_{\Pi}\left\Vert \S_{\Pi}w\right\Vert _{\Ls^{2}\left(\varSigma_{K,L}\right)}^{2},
\]
\[
\left\Vert \Q_{\Pi}w\right\Vert _{\Ls^{2}\left(\varSigma_{K,\sigma}\right)}^{2}\leq C_{\Pi}\left\Vert \S_{\Pi}w\right\Vert _{\Ls^{2}\left(\varSigma_{K,\sigma}\right)}^{2}
\]
over all cells of the dual mesh concludes the proof.
\end{proof}

\section{Convergence}

The existence, uniqueness of the weak solution and convergence of
the numerical scheme will be shown using a single procedure that relies
on an a priori estimate to ensure boundedness of the respective numerical
solutions. This estimate is independent of mesh refinement. The procedure
uses and expands upon some of the ideas presented in \cite{Benes-Math_comp_aspects_solid}
and \cite{Benes-Asymptotics}. For the sake of simplicity, the homogeneous
Dirichlet boundary conditions 
\begin{align}
\begin{array}{c}
u_{\partial\Omega}=0,\end{array} & \text{ }p_{\partial\Omega}=0\label{eq:ZeroDirichletBC}
\end{align}
are considered, allowing to simplify (\ref{eq:flux-term}) and (\ref{eq:flux-mesh-function})
to
\begin{align*}
F_{K,\sigma}\left(w_{\Pi}\right) & \equiv F_{K,\sigma}\left(w_{\Pi},0\right),\quad F_{K}\left(w_{\Pi}\right)\equiv F_{K}\left(w_{\Pi},0\right),\\
F\left(w_{\Pi}\right) & \equiv F\left(w_{\Pi},0\right)\qquad\forall w\in\H^{\Pi}
\end{align*}

\subsection{Weak Formulation}

Testing the equations (\ref{eq:PhaseIsoHeat-1})--(\ref{eq:PhaseIso-1})
by $\psi$, $\varphi\in\Cs_{0}^{\infty}\left(\mathcal{J}\right)$
and $v$, $q\in\Cs_{0}^{\infty}\left(\Omega\right)$, the weak formulation

\begin{align}
\stackrel[0]{T}{\int}\frac{\dd}{\dd t}\underset{\Omega}{\int}u\left(t,\vec{x}\right)v\left(\vec{x}\right)\dd\vec{x}\psi\left(t\right)\dd t & +\stackrel[0]{T}{\int}\underset{\Omega}{\int}\nabla u\left(t,\vec{x}\right)\cdot\nabla v\left(\vec{x}\right)\dd\vec{x}\psi\left(t\right)\dd t\nonumber \\
 & =\stackrel[0]{T}{\int}\frac{\dd}{\dd t}\underset{\Omega}{\int}Lp\left(t,\vec{x}\right)v\left(\vec{x}\right)\dd\vec{x}\psi\left(t\right)\dd t,\label{eq:weakFormulation-u}\\
\alpha\xi^{2}\stackrel[0]{T}{\int}\frac{\dd}{\dd t}\underset{\Omega}{\int}p\left(t,\vec{x}\right)q\left(\vec{x}\right)\dd\vec{x}\varphi\left(t\right)\dd t & +\xi^{2}\stackrel[0]{T}{\int}\underset{\Omega}{\int}\nabla p\left(t,\vec{x}\right)\cdot\nabla q\left(\vec{x}\right)\dd\vec{x}\varphi\left(t\right)\dd t\nonumber \\
 & =\stackrel[0]{T}{\int}\underset{\Omega}{\int}f\left(u,p,\nabla p;\xi\right)q\left(\vec{x}\right)\dd\vec{x}\varphi\left(t\right)\dd t,\label{eq:weakFormulation-p}
\end{align}
arises, completed by the initial conditions

\[
u\mid_{t=0}=u_{\text{ini}},p\mid_{t=0}=p_{\text{ini}}.\tag{{\ref{eq:initIso-1}}}
\]

\begin{defn}
\label{def:weak-solution}A pair of functions $\left(u,p\right)$,
such that $u\in\Ls^{2}\left(\J;\Hs_{0}^{1}\left(\Omega\right)\right),p\in\Ls^{2}\left(\J;\Hs_{0}^{1}\left(\Omega\right)\right)$
that satisfy (\ref{eq:weakFormulation-u})--(\ref{eq:weakFormulation-p})
with the initial conditions (\ref{eq:initIso-1}) for all $\psi,\varphi\in\Cs_{0}^{\infty}\left(\mathcal{J}\right)$
and all $v,q\in\Cs_{0}^{\infty}\left(\Omega\right)$ is called the
weak solution of the problem (\ref{eq:PhaseIsoHeat-1})--(\ref{eq:initIso-1})
with the boundary conditions (\ref{eq:ZeroDirichletBC}).
\end{defn}

\subsection{A priori Estimates}

Let $\Pi$ be an admissible mesh. Consider the semi-discrete scheme
(\ref{eq:semiDisSchemeHeat}), (\ref{eq:semiDisSchemePhase}). Multiplying
equation (\ref{eq:semiDisSchemeHeat}) by $\dot{u}_{K}\left(t\right)$
and equation (\ref{eq:semiDisSchemePhase}) by $\dot{p}_{K}\left(t\right)$
respectively, summing each of them over all $K\in\Pi$ and using the
definition of $\left(u,v\right)_{\Pi}$ yields

\begin{align}
\left\Vert \dot{u}_{\Pi}\left(t\right)\right\Vert _{\Pi}^{2}+\left(F\left(u_{\Pi}\left(t\right)\right),\dot{u}_{\Pi}\left(t\right)\right)_{\Pi} & =L\left(\dot{p}_{\Pi}\left(t\right),\dot{u}_{\Pi}\left(t\right)\right)_{\Pi},\label{eq:(1a)}\\
\alpha\xi^{2}\left\Vert \dot{p}_{\Pi}\left(t\right)\right\Vert _{\Pi}^{2}+\xi^{2}\left(F\left(p_{\Pi}\left(t\right)\right),\dot{p}_{\Pi}\left(t\right)\right)_{\Pi} & =\overbrace{\underset{K\in\Pi}{\sum}m\left(K\right)f_{0}\left(p_{K}\left(t\right)\right)\dot{p}_{K}\left(t\right)}^{\text{I}.}\label{eq:(2a)}\\
 & -b\beta\xi\left(\varLambda_{\Pi}\left(t\right),\dot{p}_{K}\left(t\right)\right)_{\Pi}.\nonumber 
\end{align}
Using the chain rule and the fact that $f_{0}$ is the derivative
of the double well potential $w_{0}$ (see \cite{Caginalp-Stefan_HeleShaw_PF,Benes-Math_comp_aspects_solid}
and the proof of Lemma \ref{lem:doublewell-pot-estimate}), we obtain

\begin{equation}
f_{0}\left(p\left(t\right)\right)\dot{p}_{K}\left(t\right)=-\frac{\dd w_{0}}{\dd t}\left(p_{K}(t)\right).\label{eq:doubleMinAndF0}
\end{equation}
Rewriting expression I. using (\ref{eq:doubleMinAndF0}), we get

\begin{align}
\underset{K\in\Pi}{\sum}m\left(K\right)f_{0}\left(p_{K}\left(t\right)\right)\dot{p}_{K}\left(t\right) & =\underset{K\in\Pi}{\sum}-\frac{\dd w_{0}}{\dd t}\left(p_{K}(t)\right)m(K).\label{eq:fTow1}
\end{align}
The Schwarz and Young inequalities applied on the right hand side
of (\ref{eq:(1a)}) and (\ref{eq:(2a)}) together with the boundedness
of $\Lambda$ (bounded by the constant $B$) and (\ref{eq:fTow1})
give

\begin{align}
\frac{1}{2}\left\Vert \dot{u}_{\Pi}\left(t\right)\right\Vert _{\Pi}^{2}+\left(F\left(u_{\Pi}\left(t\right)\right),\dot{u}_{\Pi}\left(t\right)\right)_{\Pi} & \leq\frac{L^{2}}{2}\left\Vert \dot{p}_{\Pi}\left(t\right)\right\Vert _{\Pi}^{2},\label{eq:(1b)}\\
\frac{1}{2}\alpha\xi^{2}\left\Vert \dot{p}_{\Pi}\left(t\right)\right\Vert _{\Pi}^{2}+\xi^{2}\left(F\left(p_{\Pi}\left(t\right)\right),\dot{p}_{\Pi}\left(t\right)\right)_{\Pi} & +\frac{\dd}{\dd t}\underset{K\in\Pi}{\sum}w_{0}\left(p_{K}(t)\right)m(K)\label{eq:(2b)}\\
\leq & \frac{\left(b\beta\right)^{2}}{2\alpha}B^{2}m\left(\Omega\right).\nonumber 
\end{align}
Multiplying (\ref{eq:(1b)}) by $\frac{\alpha\xi^{2}}{2\Ls^{2}}$
and adding it to (\ref{eq:(2b)}) results in a single inequality

\begin{align}
\frac{1}{4}\alpha\xi^{2}\left\Vert \dot{p}_{\Pi}\left(t\right)\right\Vert _{\Pi}^{2}+\frac{1}{4}\frac{\alpha\xi^{2}}{2\Ls^{2}}\left\Vert \dot{u}_{\Pi}\left(t\right)\right\Vert _{\Pi}^{2} & +\xi^{2}\overbrace{\left(F\left(p_{\Pi}\left(t\right)\right),\dot{p}_{\Pi}\left(t\right)\right)_{\Pi}}^{\text{I.}}\label{eq:3a}\\
+\frac{\alpha\xi^{2}}{2\Ls^{2}}\overbrace{\left(F\left(u_{\Pi}\left(t\right)\right),\dot{u}_{\Pi}\left(t\right)\right)_{\Pi}}^{\text{II.}}+\frac{\dd}{\dd t}\underset{K\in\Pi}{\sum}w_{0}\left(p_{K}(t)\right)m(K) & \leq\frac{\left(b\beta\right)^{2}}{2\alpha}B^{2}m\left(\Omega\right).\nonumber 
\end{align}

We reformulate the terms I. and II. Under the assumption (\ref{eq:ZeroDirichletBC}),
the term I. can be rewritten by summation over faces $\sigma\in\mathcal{E}$
instead of cells $K\in\Pi$ as follows

\begin{align}
\text{I.}= & \underset{K\in\Pi}{\sum}\left(\underset{\sigma\in\E_{K}}{\sum}F_{K,\sigma}\left(p_{\Pi}\left(t\right)\right)\right)\dot{p}_{\Pi}\left(t\right)\label{eq:reformulation I.}\\
= & \underset{K\in\Pi}{\sum}\underset{{\sigma\in\E_{K}\cap\E_{\text{int}}\atop \sigma=K|L}}{\sum}-\tau_{\sigma}\left(p_{L}\left(t\right)-p_{K}\left(t\right)\right)\dot{p}_{k}\left(t\right)\nonumber \\
+ & \underset{K\in\Pi}{\sum}\underset{\sigma\in\E_{K}\cap\E_{\text{ext}}}{\sum}-\tau_{\sigma}\left(-p_{K}\left(t\right)\right)\dot{p}_{k}\left(t\right)\nonumber \\
= & -\underset{{\sigma\in\E_{\text{int}}\atop \sigma=K\left|L\right.}}{\sum}\tau_{\sigma}\left[\left(p_{L}\left(t\right)-p_{K}\left(t\right)\right)\dot{p}_{K}\left(t\right)+\left(p_{K}\left(t\right)-p_{L}\left(t\right)\right)\dot{p}_{L}\left(t\right)\right]\nonumber \\
+ & \underset{{\sigma\in\E_{\text{ext}}\atop \sigma\in\mathcal{E}_{K}}}{\sum}\tau_{\sigma}p_{K}\left(t\right)\dot{p}_{K}\left(t\right)\nonumber \\
= & \frac{1}{2}\underset{{\sigma\in\E_{\text{int}}\atop \sigma=K\left|L\right.}}{\sum}\tau_{\sigma}\left[\frac{\dd}{\dd t}\left(p_{K}\left(t\right)\right)^{2}-2\frac{\dd}{\dd t}\left(p_{L}\left(t\right)p_{K}\left(t\right)\right)+\frac{\dd}{\dd t}\left(p_{L}\left(t\right)\right)^{2}\right]\nonumber \\
+ & \frac{1}{2}\underset{{\sigma\in\E_{\text{ext}}\atop \sigma\in\mathcal{E}_{K}}}{\sum}\tau_{\sigma}\frac{\dd}{\dd t}\left(p_{K}\left(t\right)\right)^{2}\nonumber \\
= & \frac{1}{2}\frac{\dd}{\dd t}\underset{{\sigma\in\E_{\text{int}}\atop \sigma=K\left|L\right.}}{\sum}\tau_{\sigma}\left[p_{K}\left(t\right)-p_{L}\left(t\right)\right]^{2}+\frac{1}{2}\frac{\dd}{\dd t}\underset{{\sigma\in\E_{\text{ext}}\atop \sigma\in\mathcal{E}_{K}}}{\sum}\tau_{\sigma}\left(p_{K}\left(t\right)\right)^{2}.\nonumber 
\end{align}
An analogous calculation performed on term II. together with Definition
\ref{def:inner-product-on-mesh} give

\begin{equation}
\text{I.}=\frac{1}{2}\frac{\dd}{\dd t}\left\llbracket u_{\Pi}\left(t\right)\right\rrbracket _{\Pi}^{2},\qquad\text{II.}=\frac{1}{2}\frac{\dd}{\dd t}\left\llbracket p_{\Pi}\left(t\right)\right\rrbracket _{\Pi}^{2}.\label{eq:reformulation I_II}
\end{equation}
The identities (\ref{eq:reformulation I_II}) make it possible to
write (\ref{eq:3a}) in the form

\begin{align}
 & \frac{1}{4}\alpha\xi^{2}\left\Vert \dot{p}_{\Pi}\left(t\right)\right\Vert _{\Pi}^{2}+\frac{1}{4}\frac{\alpha\xi^{2}}{2\Ls^{2}}\left\Vert \dot{u_{\Pi}}\left(t\right)\right\Vert _{\Pi}^{2}+\frac{1}{2}\xi^{2}\frac{\dd}{\dd t}\left\llbracket p_{\Pi}\left(t\right)\right\rrbracket _{\Pi}^{2}\nonumber \\
+ & \frac{1}{2}\frac{\alpha\xi^{2}}{2\Ls^{2}}\frac{\dd}{\dd t}\left\llbracket u_{\Pi}\left(t\right)\right\rrbracket _{\Pi}^{2}+\frac{\dd}{\dd t}\underset{K\in\Pi}{\sum}w_{0}\left(p_{K}(t)\right)m(K)\label{eq:3b}\\
\leq & \frac{\left(b\beta\right)^{2}}{2\alpha}B^{2}m\left(\Omega\right).\nonumber 
\end{align}
At this point, the inequality (\ref{eq:3b}) will be used in two different
ways. The first result will be used within the derivation of the second
to obtain the final estimate. First, the nonnegative terms $\left\Vert \dot{p}_{\Pi}\left(t\right)\right\Vert _{\Pi}^{2}$
and $\left\Vert \dot{u_{\Pi}}\left(t\right)\right\Vert _{\Pi}^{2}$
are omitted from the left hand side of (\ref{eq:3b}) and the nonnegative
expression

\begin{equation}
\xi^{2}\frac{1}{2}\left\llbracket p_{\Pi}\left(t\right)\right\rrbracket _{\Pi}^{2}+\frac{1}{2}\frac{\alpha\xi^{2}}{2\Ls^{2}}\left\llbracket u_{\Pi}\left(t\right)\right\rrbracket _{\Pi}^{2}+\underset{K\in\Pi}{\sum}w_{0}\left(p_{K}(t)\right)m(K)\label{eq:posQadd}
\end{equation}
is added to the right hand side of (\ref{eq:3b}), giving rise to

\begin{align}
 & \frac{1}{2}\xi^{2}\frac{\dd}{\dd t}\left\llbracket p_{\Pi}\left(t\right)\right\rrbracket _{\Pi}^{2}+\frac{1}{2}\frac{\alpha\xi^{2}}{2\Ls^{2}}\frac{\dd}{\dd t}\left\llbracket u_{\Pi}\left(t\right)\right\rrbracket _{\Pi}^{2}+\frac{\dd}{\dd t}\underset{K\in\Pi}{\sum}w_{0}\left(p_{K}(t)\right)m(K)\nonumber \\
\leq & \xi^{2}\frac{1}{2}\left\llbracket p_{\Pi}\left(t\right)\right\rrbracket _{\Pi}^{2}+\frac{1}{2}\frac{\alpha\xi^{2}}{2\Ls^{2}}\left\llbracket u_{\Pi}\left(t\right)\right\rrbracket _{\Pi}^{2}+\underset{K\in\Pi}{\sum}w_{0}\left(p_{K}(t)\right)m(K)\nonumber \\
+ & \frac{\left(b\beta\right)^{2}}{2\alpha}m\left(\Omega\right)B^{2}.\label{eq:beforeGron2}
\end{align}
Let $s\in\mathbb{\J}$. Substituting $t$ for $s$, multiplying the
whole inequality by $\e^{-s}$ leads to

\begin{align*}
\frac{\dd}{\dd s}\left[\left(\frac{1}{2}\xi^{2}\left\llbracket p_{\Pi}\left(s\right)\right\rrbracket _{\Pi}^{2}+\frac{\alpha\xi^{2}}{4L}\left\llbracket u_{\Pi}\left(s\right)\right\rrbracket _{\Pi}^{2}+\underset{K\in\Pi}{\sum}w_{0}\left(p_{K}(s)\right)m(K)\right)\e^{-s}\right]\\
\leq\frac{\left(b\beta\right)^{2}}{2\alpha}B^{2}m\left(\Omega\right)\e^{-s}.
\end{align*}
Integrating with respect to $s$ over $\left(0,t\right)$ gives 

\begin{align}
 & \frac{1}{2}\xi^{2}\left\llbracket p_{\Pi}\left(t\right)\right\rrbracket _{\Pi}^{2}+\frac{\alpha\xi^{2}}{4L}\left\llbracket u_{\Pi}\left(t\right)\right\rrbracket _{\Pi}^{2}+\underset{K\in\Pi}{\sum}w_{0}\left(p_{K}(t)\right)m(K)\nonumber \\
\leq & \left(\frac{1}{2}\xi^{2}\left\llbracket p_{\Pi}\left(0\right)\right\rrbracket _{\Pi}^{2}+\frac{\alpha\xi^{2}}{4L}\left\llbracket u_{\Pi}\left(0\right)\right\rrbracket _{\Pi}^{2}\right)\e^{t}+\left(\underset{K\in\Pi}{\sum}w_{0}\left(p_{K}(0)\right)m(K)\right)\e^{t}\nonumber \\
+ & \frac{\left(b\beta\right)^{2}}{2\alpha}B^{2}m\left(\Omega\right)\left(\e^{t}-1\right).\label{eq:gronwallFinEst}
\end{align}
This inequality will be used as part of the next estimate.

Revisiting (\ref{eq:3b}) and adding the nonnegative expression (\ref{eq:posQadd})
to the right hand side yields

\begin{align}
 & \frac{1}{4}\alpha\xi^{2}\left\Vert \dot{p}_{\Pi}\left(t\right)\right\Vert _{\Pi}^{2}+\frac{1}{4}\frac{\alpha\xi^{2}}{2\Ls^{2}}\left\Vert \dot{u}_{\Pi}\left(t\right)\right\Vert _{\Pi}^{2}\label{eq:estimate2Start}\\
+ & \frac{1}{2}\frac{\dd}{\dd t}\left(\xi^{2}\left\llbracket p_{\Pi}\left(t\right)\right\rrbracket _{\Pi}^{2}\right)+\frac{1}{2}\frac{\dd}{\dd t}\left(\frac{\alpha\xi^{2}}{2\Ls^{2}}\left\llbracket u_{\Pi}\left(t\right)\right\rrbracket _{\Pi}^{2}\right)+\frac{\dd}{\dd t}\underset{K\in\Pi}{\sum}w_{0}\left(p_{K}(t)\right)m(K)\nonumber \\
\leq & \frac{\left(b\beta\right)^{2}}{2\alpha}B^{2}m\left(\Omega\right)+\frac{1}{2}\xi^{2}\left\llbracket p_{\Pi}\left(t\right)\right\rrbracket _{\Pi}^{2}\nonumber \\
+ & \frac{1}{2}\frac{\alpha\xi^{2}}{2\Ls^{2}}\left\llbracket u_{\Pi}\left(t\right)\right\rrbracket _{\Pi}^{2}+\underset{K\in\Pi}{\sum}w_{0}\left(p_{K}(t)\right)m(K).\nonumber 
\end{align}
Integrating this estimate with respect to $t$ over $\J=\left(0,T\right)$
gives 

\begin{align}
 & \stackrel[0]{T}{\int}\frac{1}{4}\alpha\xi^{2}\left\Vert \dot{p}_{\Pi}\left(t\right)\right\Vert _{\Pi}^{2}\dd t+\stackrel[0]{T}{\int}\frac{1}{4}\frac{\alpha\xi^{2}}{2\Ls^{2}}\left\Vert \dot{u}_{\Pi}\left(t\right)\right\Vert _{\Pi}^{2}\dd t+\frac{1}{2}\xi^{2}\left\llbracket p_{\Pi}\left(T\right)\right\rrbracket _{\Pi}^{2}\nonumber \\
+ & \frac{1}{2}\frac{\alpha\xi^{2}}{2\Ls^{2}}\left\llbracket u_{\Pi}\left(T\right)\right\rrbracket _{\Pi}^{2}+\left(\underset{K\in\Pi}{\sum}w_{0}\left(p_{K}(T)\right)m(K)\right)\nonumber \\
\leq & \frac{1}{2}\xi^{2}\left\llbracket p_{\Pi}\left(0\right)\right\rrbracket _{\Pi}^{2}+\frac{1}{2}\frac{\alpha\xi^{2}}{2\Ls^{2}}\left\llbracket u_{\Pi}\left(0\right)\right\rrbracket _{\Pi}^{2}\nonumber \\
+ & \stackrel[0]{T}{\int}\frac{1}{2}\xi^{2}\left\llbracket p_{\Pi}\left(t\right)\right\rrbracket _{\Pi}^{2}+\frac{1}{2}\frac{\alpha\xi^{2}}{2\Ls^{2}}\left\llbracket u_{\Pi}\left(t\right)\right\rrbracket _{\Pi}^{2}\dd t+\left(\underset{K\in\Pi}{\sum}w_{0}\left(p_{K}(0)\right)m(K)\right)\nonumber \\
+ & \stackrel[0]{T}{\int}\underset{K\in\Pi}{\sum}w_{0}\left(p_{K}(t)\right)m(K)\dd t+\frac{\left(b\beta\right)^{2}}{2\alpha}B^{2}m\left(\Omega\right)T.\label{eq:2ndEstALmostDone}
\end{align}
The relationship (\ref{eq:gronwallFinEst}) is used to estimate the
integral on the right hand side

\begin{align*}
 & \stackrel[0]{T}{\int}\frac{1}{2}\xi^{2}\left\llbracket p_{\Pi}\left(t\right)\right\rrbracket _{\Pi}^{2}+\frac{1}{2}\frac{\alpha\xi^{2}}{2\Ls^{2}}\left\llbracket u_{\Pi}\left(t\right)\right\rrbracket _{\Pi}^{2}\dd t+\stackrel[0]{T}{\int}\underset{K\in\Pi}{\sum}w_{0}\left(p_{K}(t)\right)m(K)\dd t\\
\leq & \stackrel[0]{T}{\int}\left\{ \frac{1}{2}\xi^{2}\left\llbracket p_{\Pi}\left(0\right)\right\rrbracket _{\Pi}^{2}\e^{t}\right.+\frac{1}{2}\frac{\alpha\xi^{2}}{2\Ls^{2}}\left\llbracket u_{\Pi}\left(0\right)\right\rrbracket _{\Pi}^{2}\e^{t}\\
+ & \left(\underset{K\in\Pi}{\sum}w_{0}\left(p_{K}(0)\right)m(K)\right)\e^{t}+\left.\frac{\left(b\beta\right)^{2}}{2\alpha}B^{2}m\left(\Omega\right)\left(\e^{t}-1\right)\right\} \dd t\\
= & \left\{ \frac{1}{2}\xi^{2}\left\llbracket p_{\Pi}\left(0\right)\right\rrbracket _{\Pi}^{2}\right.+\frac{1}{2}\frac{\alpha\xi^{2}}{2\Ls^{2}}\left\llbracket u_{\Pi}\left(0\right)\right\rrbracket _{\Pi}^{2}\\
+ & \left(\underset{K\in\Pi}{\sum}w_{0}\left(p_{K}(0)\right)m(K)\right)+\left.\frac{\left(b\beta\right)^{2}}{2\alpha}B^{2}m\left(\Omega\right)\right\} \left(\e^{T}-1\right)-\frac{\left(b\beta\right)^{2}}{2\alpha}B^{2}m\left(\Omega\right)T.
\end{align*}
Using this estimate to simplify the right hand side of (\ref{eq:2ndEstALmostDone})
results in

\begin{align}
 & \stackrel[0]{T}{\int}\frac{1}{4}\alpha\xi^{2}\left\Vert \dot{p}_{\Pi}\left(t\right)\right\Vert _{\Pi}^{2}\dd t+\stackrel[0]{T}{\int}\frac{1}{4}\frac{\alpha\xi^{2}}{2\Ls^{2}}\left\Vert \dot{u}_{\Pi}\left(t\right)\right\Vert _{\Pi}^{2}\dd t+\frac{1}{2}\xi^{2}\left\llbracket p_{\Pi}\left(T\right)\right\rrbracket _{\Pi}^{2}\nonumber \\
+ & \frac{1}{2}\frac{\alpha\xi^{2}}{2\Ls^{2}}\left\llbracket u_{\Pi}\left(T\right)\right\rrbracket _{\Pi}^{2}+\left(\underset{K\in\Pi}{\sum}w_{0}\left(p_{K}(T)\right)m(K)\right)\nonumber \\
\leq & \left[\frac{1}{2}\xi^{2}\left\llbracket p_{\Pi}\left(0\right)\right\rrbracket _{\Pi}^{2}+\frac{\alpha\xi^{2}}{2\Ls^{2}}\left\llbracket u_{\Pi}\left(0\right)\right\rrbracket _{\Pi}^{2}+\left(\underset{K\in\Pi}{\sum}w_{0}\left(p_{K}(0)\right)m(K)\right)\right]\e^{T}\nonumber \\
+ & \frac{\left(b\beta\right)^{2}}{2\alpha}B^{2}m\left(\Omega\right)\left(\e^{T}-1\right).\label{eq:2ndEstOneBeforeFin}
\end{align}

\begin{lem}
\label{lem:doublewell-pot-estimate}There exists a constant $c_{w}>0$
such that

\begin{equation}
w_{0}\left(x\right)+c_{w}\geq x^{2}\qquad\forall x\in\R.\label{eq:LemmaOmega}
\end{equation}
\end{lem}
\begin{proof}
\noindent Consider the relationship between $f_{0}$ and $w_{0}$
\[
w_{0}\left(x\right)=-\stackrel[0]{x}{\int}f_{0}\left(s\right)\dd s=\frac{1}{4}x^{2}\left(x-1\right)^{2}.
\]
We rewrite the expression
\begin{equation}
w_{0}\left(x\right)-x^{2}=\frac{1}{4}x^{2}\left[\left(x-1\right)^{2}-4\right]=\frac{1}{4}x^{2}\left(x+1\right)\left(x-3\right),\label{eq:w0-estimate}
\end{equation}
which shows that for any $x>3$ or $x<-1$, (\ref{eq:w0-estimate})
is positive. Hence the inequality (\ref{eq:LemmaOmega}) holds for
any $c_{w}\geq0$ when $x\in\left(-\infty,-1\right)\cup\left(3,+\infty\right)$.
We extend this inequality to any $x\in\mathbb{R}$ by setting
\[
c_{w}\equiv-\underset{x\in\left[-1,3\right]}{\text{min}}\frac{1}{4}x^{2}\left(x+1\right)\left(x-3\right).
\]
\end{proof}
We apply Lemma \ref{lem:doublewell-pot-estimate} by estimating
\begin{align}
\underset{K\in\Pi}{\sum}w_{0}\left(p_{K}(t)\right)m(K) & \geq\underset{K\in\Pi}{\sum}p_{K}^{2}(t)m(K)-\left(c_{w},1\right)_{\Pi}\nonumber \\
 & =\left\Vert \left(p_{\Pi}\left(t\right)\right)\right\Vert _{\Pi}^{2}-\left(c_{w},1\right)_{\Pi}\nonumber \\
 & \geq-\left(c_{w},1\right)_{\Pi}.\label{eq:w0-estimate2}
\end{align}
Using (\ref{eq:w0-estimate2}) and Lemma \ref{lem:discrete-to-L2-relationships},
the estimate (\ref{eq:2ndEstOneBeforeFin}) can be rewritten as

\begin{align}
 & \stackrel[0]{T}{\int}\frac{1}{4}\alpha\xi^{2}\left\Vert \S_{\Pi}\dot{p}_{\Pi}\left(t\right)\right\Vert _{\Ls^{2}\left(\Omega\right)}^{2}\dd t+\stackrel[0]{T}{\int}\frac{1}{4}\frac{\alpha\xi^{2}}{2\Ls^{2}}\left\Vert \S_{\Pi}\dot{u}_{\Pi}\left(t\right)\right\Vert _{\Ls^{2}\left(\Omega\right)}^{2}\dd t\nonumber \\
+ & \frac{3}{2}\xi^{2}\left\Vert \left|\nabla\mathcal{Q}_{\Pi}p_{\Pi}\left(T\right)\right|\right\Vert _{\Ls^{2}\left(\Omega\right)}^{2}+\frac{3}{2}\frac{\alpha\xi^{2}}{2\Ls^{2}}\left\Vert \left|\nabla\mathcal{Q}_{\Pi}u_{\Pi}\right|\right\Vert _{\Ls^{2}\left(\Omega\right)}^{2}\nonumber \\
\leq & \left[\frac{1}{2}\xi^{2}\left\llbracket p_{\Pi}\left(0\right)\right\rrbracket _{\Pi}^{2}+\frac{\alpha\xi^{2}}{2\Ls^{2}}\left\llbracket u_{\Pi}\left(0\right)\right\rrbracket _{\Pi}^{2}+\left(\underset{K\in\Pi}{\sum}w_{0}\left(p_{K}(0)\right)m(K)\right)\right]\e^{T}\label{eq:FINALaprioriEST}\\
+ & \frac{\left(b\beta\right)^{2}}{2\alpha}B^{2}m\left(\Omega\right)\left(\e^{T}-1\right)+\left(c_{w},1\right)_{\Pi}.\nonumber 
\end{align}

\subsection{Convergence of the Numerical Solution}

All the quantities on the left hand side of (\ref{eq:FINALaprioriEST})
are nonnegative, thus the left hand side is bounded from below. The
estimate (\ref{eq:FINALaprioriEST}) shows that the left hand side
is also bounded from above. This implies that all of the expressions
on the left hand side of (\ref{eq:FINALaprioriEST}) are bounded.
Interpreting this boundedness in the context of Bochner spaces, we
get

\begin{align}
\frac{\partial}{\partial t}\S_{\Pi}p_{\Pi} & =\S_{\Pi}\dot{p}_{\Pi}\in\Ls^{2}\left(\J;\Ls^{2}\left(\Omega\right)\right),\label{eq:derInterpo1}\\
\frac{\partial}{\partial t}\S_{\Pi}u_{\Pi} & =\S_{\Pi}\dot{u}_{\Pi}\in\Ls^{2}\left(\J;\Ls^{2}\left(\Omega\right)\right).\label{eq:derInterpo2}
\end{align}
Furthermore, using Lemma \ref{lem:Q-S-estimate}, it also holds that

\begin{align*}
\frac{\partial}{\partial t}\Q_{\Pi}p_{\Pi} & \in\Ls^{2}\left(\J;\Ls^{2}\left(\Omega\right)\right),\\
\frac{\partial}{\partial t}\Q_{\Pi}u_{\Pi} & \in\Ls^{2}\left(\J;\Ls^{2}\left(\Omega\right)\right).
\end{align*}
The finite Bochner norms together with the boundedness of $\Omega$
and $\J$ imply essential boundedness%
\begin{lyxgreyedout}
Ales: Extra proof 1 (??)%
\end{lyxgreyedout}

\begin{align*}
\S_{\Pi}p_{\Pi}\in L^{\infty}\left(\J;\Ls^{2}\left(\Omega\right)\right) & ,\Q_{\Pi}p_{\Pi}\in L^{\infty}\left(\J;\Ls^{2}\left(\Omega\right)\right),\\
\S_{\Pi}u_{\Pi}\in L^{\infty}\left(\J;\Ls^{2}\left(\Omega\right)\right), & \Q_{\Pi}u_{\Pi}\in L^{\infty}\left(\J;\Ls^{2}\left(\Omega\right)\right).
\end{align*}
Finally, the estimate (\ref{eq:FINALaprioriEST}) also gives%
\begin{lyxgreyedout}
\textbf{OMITTED:}
\begin{align*}
\frac{\partial}{\partial x^{i}}\Q_{\Pi}p_{\Pi} & \in L^{\infty}\left(\J;\Ls^{2}\left(\Omega\right)\right),i\in\left\{ 1,2,3\right\} ,\\
\frac{\partial}{\partial x^{i}}\Q_{\Pi}u_{\Pi} & \in L^{\infty}\left(\J;\Ls^{2}\left(\Omega\right)\right),i\in\left\{ 1,2,3\right\} .
\end{align*}
\end{lyxgreyedout}

\begin{align}
\Q_{\Pi}p_{\Pi} & \in L^{\infty}\left(\J;\Hs_{0}^{1}\left(\Omega\right)\right),\label{eq:interpolationPinSobo}\\
\Q_{\Pi}u_{\Pi} & \in L^{\infty}\left(\J;\Hs_{0}^{1}\left(\Omega\right)\right).\label{eq:interpolationUinSobo}
\end{align}
In order to facilitate the subsequent analysis, we introduce the concept
of a normal sequence of meshes.
\begin{defn}
\label{normal-mesh-seq}The norm of an admissible mesh $\Pi$ is defined
as
\[
\left|\Pi\right|\equiv\underset{K\in\Pi}{\text{max}}\,\text{\ensuremath{\inf}}\left\{ r>0\left|\exists x_{0}\in\mathbb{R}^{3}:\forall x\in K:\left|x-x_{0}\right|<r\right.\right\} .
\]
A sequence of admissible meshes $\left(\Pi_{n}\right)$ is called
\emph{normal} if and only if $\underset{n\rightarrow+\infty}{\text{lim}}\left|\Pi_{n}\right|=0$.%
\begin{lyxgreyedout}
This is our ``original'' term based on the normal sequence of partitions
in the construction of Riemann integral.%
\end{lyxgreyedout}
\end{defn}
\begin{lem}
\label{lem:interp-S}Let $w\in\Ls^{2}\left(\Omega\right)$ and $\Pi_{n}$
be a normal sequence of admissible meshes. Then

\begin{equation}
\S_{\Pi_{n}}\P_{\Pi_{n}}w\rightarrow w\text{ in }\Ls^{2}\left(\Omega\right).\label{eq:interpolationToOrg}
\end{equation}
\end{lem}
Lemma \ref{lem:interp-S} is a consequence of the definition of the
operators $\S_{\Pi_{n}},\P_{\Pi_{n}}$, and Lebesgue integration theory
\cite{Kufner-Function-Spaces}.
\begin{lem}
\label{lem:convergence-to-p}Let $u_{\text{ini}},p_{\text{ini}}\in\Cs^{2}\left(\Omega\right)$
and let $\left(\Pi_{n}\right)$ be a normal sequence of admissible
meshes. Then there exists an increasing sequence $\left(k_{n}\right)\subset\mathbb{N}$
and functions $p,u\in\Ls^{2}\left(\J;\Hs_{0}^{1}\left(\Omega\right)\right)$
with the derivatives $\frac{\partial p}{\partial t},\frac{\partial u}{\partial t}\in\Ls^{2}\left(\J;\Ls^{2}\left(\Omega\right)\right)$
such that for $n\rightarrow+\infty$, the following holds:

\begin{align}
\Q_{\Pi_{k_{n}}}p_{\Pi_{k_{n}}} & \longrightarrow p,\label{eq:conv1}\\
\S_{\Pi_{k_{n}}}p_{\Pi_{k_{n}}} & \longrightarrow p,\label{eq:conv2}\\
\S_{\Pi_{k_{n}}}\dot{p}_{\Pi_{k_{n}}} & \overset{w}{\longrightarrow}\frac{\partial p}{\partial t},\label{eq:conv3}\\
\Q_{\Pi_{k_{n}}}u_{\Pi_{k_{n}}} & \longrightarrow u,\label{eq:conv4}\\
\S_{\Pi_{k_{n}}}u_{\Pi_{k_{n}}} & \longrightarrow u,\label{eq:conv5}\\
\Q_{\Pi_{k_{n}}}\dot{u}_{\Pi_{k_{n}}} & \overset{w}{\longrightarrow}\frac{\partial u}{\partial t}\label{eq:conv6}
\end{align}
in $\Ls^{2}\left(\J;\Ls^{2}\left(\Omega\right)\right).$
\end{lem}
\begin{proof}
For each admissible mesh $\Pi_{n}$, the solutions of the semi-discrete
problem (\ref{eq:semiDisSchemeHeat}) and (\ref{eq:semiDisSchemePhase})
are $u_{\Pi_{n}}$ and $p_{\Pi_{n}}$. Let us recall the right hand
side of (\ref{eq:FINALaprioriEST}) and label the terms as follows:
\begin{align}
\overbrace{\frac{1}{2}\xi^{2}\left\llbracket p_{\Pi_{n}}\left(0\right)\right\rrbracket _{\Pi}^{2}}^{\text{I.}} & +\overbrace{\frac{\alpha\xi^{2}}{2\Ls^{2}}\left\llbracket u_{\Pi_{n}}\left(0\right)\right\rrbracket _{\Pi}^{2}}^{\text{II.}}\label{eq:RHS of final apriori est}\\
\overbrace{\left(\underset{K\in\Pi_{n}}{\sum}w_{0}\left(p_{K}(0)\right)m(K)\right)\e^{T}}^{\text{III.}} & +\overbrace{\frac{\left(b\beta\right)^{2}}{2\alpha}m\left(\Omega\right)B^{2}\left(\e^{T}-1\right)+\left(c_{w},1\right)_{\Pi_{n}}.}^{\text{IV.}}\nonumber 
\end{align}
We show that (\ref{eq:RHS of final apriori est}) is uniformly bounded
with respect to $n$. IV. is just a constant and does not depend on
$n$. The assumption $u_{\text{ini}},p_{\text{ini}}\in\Cs^{2}\left(\Omega\right)$
implies the uniform boundedness of $p_{\Pi}\left(0\right)$ w.r.t.
$n$. Since $w_{0}$ is a continuous function, the whole term III.
is bounded. Terms I. and II. are treated in the same way and so the
procedure will only be shown for term I. Thanks to (\ref{eq:ZeroDirichletBC}),
we have
\begin{align*}
\left\llbracket p_{\Pi_{n}}\left(0\right)\right\rrbracket _{n,\text{ext}}^{2} & =0.
\end{align*}
Taking this into account, we can estimate I. as follows:
\begin{align*}
\left|\left\llbracket p_{\Pi_{n}}\left(0\right)\right\rrbracket _{n,\text{int}}^{2}\right| & \equiv\left|\underset{\sigma\in\E_{n,\text{int}}}{\sum}\frac{m\left(\sigma\right)}{d_{\sigma}}\left[p_{K}\left(0\right)-p_{L}\left(0\right)\right]^{2}\right|\\
\leq\underset{\sigma\in\E_{n,\text{int}}}{\sum}m\left(\sigma\right)d_{\sigma}\left|\left[\frac{p_{K}\left(0\right)-p_{L}\left(0\right)}{d_{\sigma}}\right]^{2}\right| & \leq3m\left(\Omega\right)C_{\partial}^{2}.
\end{align*}
where $C_{\partial}$ is a constant that bounds the difference quotient
$\frac{p_{K}\left(0\right)-p_{L}\left(0\right)}{d_{\sigma}}$ independently
of $n$.

Since all of the terms on the right hand side of (\ref{eq:FINALaprioriEST})
are uniformly bounded in $n$, the left hand side must also be bounded
(the left hand side is nonnegative). This implies
\begin{itemize}
\item $\left(\S_{\Pi_{n}}\dot{p}_{\Pi_{n}}\right),\left(\Q_{\Pi_{n}}\dot{p}_{\Pi_{n}}\right),\left(\S_{\Pi_{n}}\dot{u}_{\Pi_{n}}\right)\text{ and }\left(\Q_{\Pi_{n}}\dot{u}_{\Pi_{n}}\right)$
are uniformly bounded in $\Ls^{2}\left(\J;\Ls^{2}\left(\Omega\right)\right),$
\item $\left(\S_{\Pi_{n}}p_{\Pi_{n}}\right)$ and $\left(\S_{\Pi_{n}}p_{\Pi_{n}}\right)$
are uniformly bounded in $L^{\infty}\left(\J;\Ls^{2}\left(\Omega\right)\right)$,
\item $\left(\Q_{\Pi_{n}}p_{\Pi_{n}}\right)$ and $\left(\Q_{\Pi_{n}}u_{\Pi_{n}}\right)$
are uniformly bounded in $L^{\infty}\left(\J;\Hs_{0}^{1}\left(\Omega\right)\right)$.
\end{itemize}
Since the inclusions $L^{\infty}\left(\J;\Ls^{2}\left(\Omega\right)\right)\subset\Ls^{2}\left(\J;\Ls^{2}\left(\Omega\right)\right)$
and $L^{\infty}\left(\J;\Hs_{0}^{1}\left(\Omega\right)\right)\subset\Ls^{2}\left(\J;\Hs_{0}^{1}\left(\Omega\right)\right)$
hold%
\begin{lyxgreyedout}
\textbf{OMITTED:} ``and $\Ls^{2}\left(\J;\Ls^{2}\left(\Omega\right)\right),\Ls^{2}\left(\J;\Hs_{0}^{1}\left(\Omega\right)\right)$
are both REFLEXIVE! spaces \cite{Kufner-Function-Spaces}''%
\end{lyxgreyedout}
, a weakly convergent subsequence for each of the sequences exists.
Let $\left(h_{n}\right)$ be the sequence for which 
\[
\left(\S_{\Pi_{n}}\dot{p}_{\Pi_{n}}\right),\left(\Q_{\Pi_{n}}\dot{p}_{\Pi_{n}}\right),\left(\S_{\Pi_{n}}p_{\Pi_{n}}\right),\left(\Q_{\Pi_{n}}p_{\Pi_{n}}\right)
\]
 are weakly convergent. From $\left(h_{n}\right)$ we choose $\left(k_{n}\right)$
so that in addition to this 
\[
\left(\S_{\Pi_{n}}\dot{u}_{\Pi_{n}}\right),\left(\Q_{\Pi_{n}}\dot{u}_{\Pi_{n}}\right),\left(\S_{\Pi_{n}}u_{\Pi_{n}}\right),\left(\Q_{\Pi_{n}}u_{\Pi_{n}}\right)
\]
 are weakly convergent. To simplify notation, we will use $\Pi_{n}$
to denote $\Pi_{k_{n}}$ in the following. Altogether, we may write
\begin{itemize}
\item $\left(\S_{\Pi_{n}}\dot{p}_{\Pi_{n}}\right),\left(\Q_{\Pi_{n}}\dot{p}_{\Pi_{n}}\right),\left(\S_{\Pi_{n}}\dot{u}_{\Pi_{n}}\right)\text{ and }\left(\Q_{\Pi_{n}}\dot{u}_{\Pi_{n}}\right)$
are weakly convergent in $\Ls^{2}\left(\J;\Ls^{2}\left(\Omega\right)\right),$
\item $\left(\S_{\Pi_{n}}p_{\Pi_{n}}\right)$ and $\left(\S_{\Pi_{n}}u_{\Pi_{n}}\right)$
are weakly convergent in $\Ls^{2}\left(\J;\Ls^{2}\left(\Omega\right)\right)$,
\item $\left(\Q_{\Pi_{n}}p_{\Pi_{n}}\right)$ and $\left(\Q_{\Pi_{n}}u_{\Pi_{n}}\right)$
are weakly convergent in $\Ls^{2}\left(\J;\Hs_{0}^{1}\left(\Omega\right)\right)$.
\end{itemize}
Using the compact embedding results from \cite{Kufner-Function-Spaces,Schwartz-distributions},
we get

\[
\Ls^{2}\left(\J;\Hs_{0}^{1}\left(\Omega\right)\right)\hookrightarrow\hookrightarrow\Ls^{2}\left(\J;\Ls^{2}\left(\Omega\right)\right).
\]
Hence, the strong convergence of $\left(\Q_{\Pi_{n}}p_{\Pi_{n}}\right)$
and $\left(\Q_{\Pi_{n}}u_{\Pi_{n}}\right)$ in $\Ls^{2}\left(\J;\Ls^{2}\left(\Omega\right)\right)$
follows. The definition of the interpolation operators gives
\begin{align*}
\left\Vert \S_{\Pi_{n}}p_{\Pi_{n}}-\Q_{\Pi_{n}}p_{\Pi_{n}}\right\Vert _{\Ls^{2}\left(\Omega\right)} & \overset{n\rightarrow+\infty}{\longrightarrow}0,\\
\left\Vert \S_{\Pi_{n}}u_{\Pi_{n}}-\Q_{\Pi_{n}}u_{\Pi_{n}}\right\Vert _{\Ls^{2}\left(\Omega\right)} & \overset{n\rightarrow+\infty}{\longrightarrow}0.
\end{align*}
We conclude that $\left(\S_{\Pi_{n}}p_{\Pi_{n}}\right)$ and $\left(\S_{\Pi_{n}}u_{\Pi_{n}}\right)$,
converge strongly to the same limit as $\left(\Q_{\Pi_{n}}p_{\Pi_{n}}\right)$
and $\left(\Q_{\Pi_{n}}u_{\Pi_{n}}\right)$ respectively, we will
denote these limits $p$ and $u$. Since the space $\Ls^{2}\left(\J;\Hs_{0}^{1}\left(\Omega\right)\right)$
is complete and (\ref{eq:interpolationPinSobo}), (\ref{eq:interpolationUinSobo})
we can conclude that $u,p\in\Ls^{2}\left(\J;\Hs_{0}^{1}\left(\Omega\right)\right).$
This gives the statements (\ref{eq:conv1}), (\ref{eq:conv2}), (\ref{eq:conv4})
and (\ref{eq:conv5}).

To prove the convergence of $\left(\S_{\Pi_{n}}\dot{p}_{\Pi_{n}}\right)$
and $\left(\S_{\Pi_{n}}\dot{u}_{\Pi_{n}}\right)$, we first use the
relationships (\ref{eq:derInterpo1}) and (\ref{eq:derInterpo2})
and the completeness of $\Ls^{2}\left(\J;\Ls^{2}\left(\Omega\right)\right)$
to see that 

\begin{equation}
\eta_{p},\eta_{u}\in\Ls^{2}\left(\J;\Ls^{2}\left(\Omega\right)\right),\label{eq:limitCOmpleate}
\end{equation}
where $\eta_{p}$ and $\eta_{u}$ are the weak limits of $\left(\S_{\Pi_{n}}\dot{p}_{\Pi_{n}}\right)$
and $\left(\S_{\Pi_{n}}\dot{u}_{\Pi_{n}}\right)$, respectively. Assume
that $\varphi\in\Cs_{0}^{\infty}\left(\J\right)$ and $\psi\in\Cs_{0}^{\infty}\left(\Omega\right)$
are arbitrary. Then

\begin{align*}
\stackrel[0]{T}{\int}\left\langle \left(\S_{\Pi_{n}}\dot{p}_{\Pi_{n}}\right),\psi\right\rangle \left(t\right)\varphi\left(t\right)\dd t & =-\stackrel[0]{T}{\int}\left\langle \left(\S_{\Pi_{n}}p_{\Pi_{n}}\right),\psi\right\rangle \left(t\right)\dot{\varphi}\left(t\right)\dd t\\
\overset{n\rightarrow+\infty}{\rightarrow}-\stackrel[0]{T}{\int}\left\langle p,\psi\right\rangle \left(t\right)\dot{\varphi}\left(t\right)\dd t & =\stackrel[0]{T}{\int}\left\langle \dot{p},\psi\right\rangle \left(t\right)\varphi\left(t\right)\dd t\\
\stackrel[0]{T}{\int}\left\langle \left(\S_{\Pi_{n}}\dot{p}_{\Pi_{n}}\right),\psi\right\rangle \left(t\right)\varphi\left(t\right)\dd t & \overset{n\rightarrow+\infty}{\rightarrow}\stackrel[0]{T}{\int}\left\langle \eta_{p},\psi\right\rangle \left(t\right)\varphi\left(t\right)\dd t
\end{align*}
The calculation above is possible due to (\ref{eq:limitCOmpleate})
and shows that $\left(\S_{\Pi_{n}}\dot{p}_{\Pi_{n}}\right)$ weakly
converges to $\eta_{p}=\frac{\partial p}{\partial t}$ in $\Ls^{2}\left(\J;\Ls^{2}\left(\Omega\right)\right).$
A similar procedure may be used to conclude that $\left(\S_{\Pi_{n}}\dot{u}_{\Pi_{n}}\right)$
converges weakly to $\eta_{u}=\frac{\partial u}{\partial t}$ in $\Ls^{2}\left(\J;\Ls^{2}\left(\Omega\right)\right).$
\end{proof}

\begin{lem}
\label{lem:gradient-convergence}Let $q\in\Cs_{0}^{\infty}\left(\varOmega\right)$
and $\Pi_{n}$ be a normal sequence of admissible meshes. Then (for
a suitable subsequence $\Pi_{k_{n}}$ denoted again as $\Pi_{n}$)

\begin{align}
\underset{K\in\Pi_{n}}{\sum}q\left(\vec{x}_{K}\right)F_{K}\left(p_{\Pi_{n}}\left(t\right)\right) & \overset{n\to+\infty}{\longrightarrow}\underset{\Omega}{\int}\nabla p\left(t,\vec{x}\right)\cdot\nabla q\left(\vec{x}\right)\dd\vec{x},\label{eq:flux-convergence-p}\\
\underset{K\in\Pi_{n}}{\sum}q\left(\vec{x}_{K}\right)F_{K}\left(u_{\Pi_{n}}\left(t\right)\right) & \overset{n\to+\infty}{\longrightarrow}\underset{\Omega}{\int}\nabla u\left(t,\vec{x}\right)\cdot\nabla q\left(\vec{x}\right)\dd\vec{x}\label{eq:flux-convergence-u}
\end{align}
in $\Ls^{2}\left(\mathcal{J}\right)$.
\end{lem}
\begin{proof}
We follow the ideas in \cite[Theorem 9.1]{FVM-Eymard}%
\begin{lyxgreyedout}
\textbf{EXPLANATION:} See revised PDF version page 47%
\end{lyxgreyedout}
{} and omit the integration with respect to $t\in\left(0,T\right)$
for better readability. The left hand side of (\ref{eq:flux-convergence-p})
can be rewritten as
\begin{equation}
\begin{aligned}T_{n}\equiv & \underset{K\in\Pi_{n}}{\sum}q\left(\vec{x}_{K}\right)F_{K}\left(p_{\Pi_{n}}\left(t\right)\right)\\
= & \underset{{\sigma\in\E_{n,\text{int}}\atop \sigma=K\left|L\right.}}{\sum}\tau_{\sigma}\left(p_{L}\left(t\right)-p_{K}\left(t\right)\right)\left(q\left(\vec{x}_{L}\right)-q\left(\vec{x}_{K}\right)\right)\\
+ & \underset{{\sigma\in\E_{n,\text{ext}}\atop \sigma\in\mathcal{E}_{K}}}{\sum}\tau_{\sigma}p_{K}\left(t\right)q\left(\vec{x}_{K}\right).
\end{aligned}
\label{eq:grad-discrete-term}
\end{equation}
In addition, consider the term
\begin{equation}
T_{n}'\equiv-\int\limits _{\Omega}\left(\mathcal{S}_{\Pi_{n}}p_{\Pi_{n}}\left(t\right)\right)\left(\vec{x}\right)\Delta q\left(\vec{x}\right)\dd\vec{x}\label{eq:grad-intermediate-term}
\end{equation}
which thanks to (\ref{eq:conv2}) converges to 
\[
-\int\limits _{\Omega}p\left(\vec{x}\right)\Delta q\left(\vec{x}\right)\dd\vec{x}=\int\limits _{\Omega}\nabla p\left(\vec{x}\right)\cdot\nabla q\left(\vec{x}\right)\dd\vec{x}
\]
as $n\to\infty$. First, we rewrite (\ref{eq:grad-intermediate-term})
as
\begin{align*}
T_{n}' & =\sum_{K\in\Pi_{n}}-p_{K}\left(t\right)\int\limits _{K}\Delta q\left(\vec{x}\right)\dd\vec{x}\\
 & =\sum_{K\in\Pi_{n}}-p_{K}\left(t\right)\sum_{\sigma\in\mathcal{E}_{K}}\int\limits _{\sigma}\nabla q\left(\vec{x}\right)\cdot\vec{n}_{K,\sigma}\dd S\\
 & =\underset{{\sigma\in\E_{n,\text{int}}\atop \sigma=K\left|L\right.}}{\sum}\left(p_{L}\left(t\right)-p_{K}\left(t\right)\right)\int\limits _{\sigma}\nabla q\left(\vec{x}\right)\cdot\vec{n}_{K,\sigma}\dd S.
\end{align*}
\begin{lyxgreyedout}
\textbf{REMARK:} Formally speaking, there is also the term 
\[
\underset{{\sigma\in\E_{n,\text{ext}}\atop \sigma\in\mathcal{E}_{K}}}{\sum}-p_{K}\left(t\right)\int\limits _{\sigma}\nabla q\left(\vec{x}\right)\cdot\vec{n}_{K,\sigma}\dd S.
\]
 in $T_{n}'$, which is however equal to zero since $q\in\Cs_{0}^{\infty}\left(\Omega\right)$
and thus not only $\left.q\right|_{\partial\Omega}=0$, but also $\left.\nabla q\right|_{\partial\Omega}=\vec{0}$.%
\end{lyxgreyedout}
The difference between (\ref{eq:grad-discrete-term}) and (\ref{eq:grad-intermediate-term})
can therefore be written as
\[
T_{n}-T_{n}'=\underset{{\sigma\in\E_{n,\text{int}}\atop \sigma=K\left|L\right.}}{\sum}m\left(\sigma\right)\left(p_{L}\left(t\right)-p_{K}\left(t\right)\right)R_{K,\sigma}+\underset{{\sigma\in\E_{n,\text{ext}}\atop \sigma\in\mathcal{E}_{K}}}{\sum}m\left(\sigma\right)p_{K}\left(t\right)R_{K,\sigma}
\]
where
\[
R_{K,\sigma}=\begin{cases}
\frac{q\left(\vec{x}_{L}\right)-q\left(\vec{x}_{K}\right)}{d_{\sigma}}-\int\limits _{\sigma}\nabla q\left(\vec{x}\right)\cdot\vec{n}_{K,\sigma}\dd S & \sigma=K|L\in\mathcal{E}_{n,\text{int}},\\
\frac{-q\left(\vec{x}_{K}\right)}{d_{\sigma}} & \sigma\in\mathcal{E}_{n,\text{ext}},\ \sigma\in\mathcal{E}_{K}.
\end{cases}
\]
Using Hölder's inequality leads to
\begin{align*}
\left|T_{n}-T_{n}'\right| & \leq\left(\underset{{\sigma\in\E_{n,\text{int}}\atop \sigma=K\left|L\right.}}{\sum}\tau_{\sigma}\left(p_{L}\left(t\right)-p_{K}\left(t\right)\right)^{2}\underset{{\sigma\in\E_{n,\text{int}}\atop \sigma=K\left|L\right.}}{\sum}m\left(\sigma\right)d_{\sigma}\left|R_{K,\sigma}\right|^{2}\right)^{1/2}\\
 & +\left(\underset{{\sigma\in\E_{n,\text{int}}\atop \sigma=K\left|L\right.}}{\sum}\tau_{\sigma}p_{K}\left(t\right)^{2}\underset{{\sigma\in\E_{n,\text{int}}\atop \sigma=K\left|L\right.}}{\sum}m\left(\sigma\right)d_{\sigma}\left|R_{K,\sigma}\right|^{2}\right)^{1/2}
\end{align*}
The regularity of $q$ allows to use the Taylor expansion to show
that
\[
\left|R_{K,\sigma}\right|\leq Cm\left(\sigma\right)
\]
for some $C>0$. By using $m\left(\sigma\right)\leq\left|\Pi_{n}\right|$
for each $\sigma\in\mathcal{E}_{n}$, the relation (\ref{eq:dual_cell_volume_sum_3D})
and Definition \ref{def:inner-product-on-mesh}, we further estimate
\begin{align*}
\left|T_{n}-T_{n}'\right| & \leq2C\left|\Pi_{n}\right|\sqrt{3m\left(\Omega\right)}\left\llbracket p_{\Pi_{n}}\left(t\right)\right\rrbracket _{\Pi}\\
 & =6C\left|\Pi_{n}\right|\sqrt{m\left(\Omega\right)}\left\Vert \left|\nabla\mathcal{Q}_{\Pi}w\right|\right\Vert _{\Ls^{2}\left(\Omega\right)},
\end{align*}
\begin{lyxgreyedout}
\textbf{EXPLANATION:}
\begin{align*}
 & \left|T_{n}-T_{n}'\right|\\
\leq & C\left|\Pi_{n}\right|\sqrt{3m\left(\Omega\right)}\left(\left(\underset{{\sigma\in\E_{n,\text{int}}\atop \sigma=K\left|L\right.}}{\sum}\tau_{\sigma}\left(p_{L}\left(t\right)-p_{K}\left(t\right)\right)^{2}\right)^{1/2}+\left(\underset{{\sigma\in\E_{n,\text{int}}\atop \sigma=K\left|L\right.}}{\sum}\tau_{\sigma}p_{K}\left(t\right)^{2}\right)^{1/2}\right)\\
\leq & 2C\left|\Pi_{n}\right|\sqrt{3m\left(\Omega\right)}\left(\underset{{\sigma\in\E_{n,\text{int}}\atop \sigma=K\left|L\right.}}{\sum}\tau_{\sigma}\left(p_{L}\left(t\right)-p_{K}\left(t\right)\right)^{2}+\underset{{\sigma\in\E_{n,\text{int}}\atop \sigma=K\left|L\right.}}{\sum}\tau_{\sigma}p_{K}\left(t\right)^{2}\right)^{1/2}\\
= & 2C\left|\Pi_{n}\right|\sqrt{3m\left(\Omega\right)}\left\llbracket p_{\Pi_{n}}\left(t\right)\right\rrbracket _{\Pi}
\end{align*}
\end{lyxgreyedout}
{} where the last equality is by Lemma \ref{lem:discrete-to-L2-relationships}.
The uniform boundedness of $\left\Vert \left|\nabla\mathcal{Q}_{\Pi}w\right|\right\Vert _{\Ls^{2}\left(\Omega\right)}$
given by the a priori estimate (\ref{eq:FINALaprioriEST}) and the
proof of Lemma \ref{lem:convergence-to-p} allows us to conclude that
\[
\lim_{n\to\infty}\left|T_{n}-T_{n}'\right|=0,
\]
which gives the first part of the statement, i.e. (\ref{eq:flux-convergence-p}).
The proof of (\ref{eq:flux-convergence-u}) is analogous.
\end{proof}
\begin{thm}
Let $u_{\text{ini}},p_{\text{ini}}\in\Cs^{2}\left(\Omega\right)$
and let $\left(\Pi_{n}\right)$ be a normal sequence of admissible
meshes. Then the sequence $\left(\S_{\Pi_{n}}u_{\Pi_{n}},\S_{\Pi_{n}}p_{\Pi n}\right)$
given by solutions of the semidiscrete scheme (\ref{eq:semiDisSchemeHeat}),
(\ref{eq:semiDisSchemePhase}) converges (as $n\rightarrow\infty$)
to the unique weak solution $\left(u,p\right)$ given by Definition
\ref{def:weak-solution} in $\Ls^{2}\left(\J;\Ls^{2}\left(\Omega\right)\right)$.
\end{thm}
\begin{proof}
Let us consider the semi-discrete scheme (\ref{eq:semiDisSchemeHeat}),
(\ref{eq:semiDisSchemePhase}) with the initial conditions (\ref{eq:semiDisSchemeInitU}),
(\ref{eq:semiDisSchemeInitP}) and homogeneous Dirichlet boundary
conditions (\ref{eq:ZeroDirichletBC})

\begin{align*}
m\left(K\right)\dot{u}_{K}\left(t\right)+F_{K}\left(u_{\Pi}\left(t\right)\right) & =Lm\left(K\right)\dot{p}_{K}\left(t\right),\forall K\in\Pi_{n}\tag{{\ref{eq:semiDisSchemeHeat}}}\\
\alpha m\left(K\right)\dot{p}_{K}\left(t\right)+F_{K}\left(p_{\Pi}\left(t\right)\right) & =\frac{1}{\xi^{2}}f_{0,K}\left(t\right)\tag{{\ref{eq:semiDisSchemePhase}}}\\
 & -\frac{b\beta}{\xi}\varLambda\left(g\left(u_{K}\left(t\right),p_{K}\left(t\right)\right)\right)m\left(K\right),\forall K\in\Pi_{n}\\
u_{\Pi_{n}}\left|_{\partial\Omega}\right.=0, & p_{\Pi_{n}}\left|_{\partial\Omega}\right.=0,\\
u_{\Pi_{n}}\left(0\right)=\P_{\Pi_{n}}u_{\text{ini}}, & p_{\Pi_{n}}\left(0\right)=\P_{\Pi_{n}}p_{\text{ini}}.
\end{align*}
The existence and uniqueness %
\begin{lyxgreyedout}
ORIGINALLY (WRONG): and independence on $\Pi_{n}$ of the solution
(on $\J$)%
\end{lyxgreyedout}
{} of the solution $\left(u_{\Pi_{n}},p_{\Pi_{n}}\right)$ on $\J$
follows directly from the theory of ordinary differential equations
\cite{Hartman-ODE} and the a priori estimate (\ref{eq:FINALaprioriEST}).
Let $q\in\Cs_{0}^{\infty}\left(\Omega\right)$ be a test function
and denote $q_{\Pi_{n}}\equiv\P_{\Pi_{n}}q,$ i.e. $q_{K}=q\left(\vec{x}_{K}\right)$
according to (\ref{eq:wK-simplified-notation}). Multiplying the equation
(\ref{eq:semiDisSchemePhase}) by $q_{K}$, summing the results over
all $K\in\Pi_{n}$ and using Definition \ref{def:inner-product-on-mesh},
we obtain

\begin{align*}
 & \alpha\xi^{2}\left(\dot{p}_{\varPi_{n}}\left(t\right),q_{\Pi_{n}}\right)_{\Pi_{n}}+\xi^{2}\underset{K\in\Pi_{n}}{\sum}\underset{\sigma\in\E_{K}}{\sum}F_{K,\sigma}\left(t\right)q_{K}\\
= & \left(f_{0}\left(p_{\Pi_{n}}\left(t\right)\right),q_{\Pi_{n}}\right)_{\Pi_{n}}-b\beta\xi\left(\varLambda\left(g\left(u_{\Pi_{n}}\left(t\right),p_{\Pi_{n}}\left(t\right)\right)\right),q_{\Pi_{n}}\right)_{\Pi_{n}}.
\end{align*}
We rewrite some of these terms using the inner product on $\Ls^{2}$ 

\begin{align}
 & \alpha\xi^{2}\left(\S_{\Pi_{n}}\dot{p}_{\Pi_{n}}\left(t\right),\S_{\Pi_{n}}q_{\Pi_{n}}\right)_{\Ls^{2}\left(\Omega\right)}\label{eq:L2normDone}\\
+ & \xi^{2}\underset{K\in\Pi_{n}}{\sum}\underset{\sigma\in\E_{K}}{\sum}F_{K,\sigma}\left(t\right)q_{K}\nonumber \\
= & \left(\S_{\Pi_{n}}f_{0}\left(p_{\Pi_{n}}\left(t\right)\right),\S_{\Pi_{n}}q_{\Pi_{n}}\right)_{\Ls^{2}\left(\Omega\right)}\nonumber \\
- & b\beta\xi\left(\S_{\Pi_{n}}\varLambda\left(g\left(u_{\Pi_{n}}\left(t\right),p_{\Pi_{n}}\left(t\right)\right)\right),\S_{\Pi_{n}}q_{\Pi_{n}}\right)_{\Ls^{2}\left(\Omega\right)}.\nonumber 
\end{align}
Consider a function $\varphi\in\Cs_{0}^{\infty}\left(\J\right)$.
Multiplying (\ref{eq:L2normDone}) by $\varphi$ and integrating over
$\J$ gives

\begin{align}
 & \alpha\xi^{2}\stackrel[0]{T}{\int}\left(\S_{\Pi_{n}}\dot{p}_{\Pi_{n}}\left(t\right),\S_{\Pi_{n}}q_{\Pi_{n}}\right)_{\Ls^{2}\left(\Omega\right)}\varphi\left(t\right)\dd t\label{eq:L2normDone-1}\\
+ & \xi^{2}\stackrel[0]{T}{\int}\underset{K\in\Pi_{n}}{\sum}\underset{\sigma\in\E_{K}}{\sum}F_{K,\sigma}\left(t\right)q_{K}\varphi\left(t\right)\dd t\nonumber \\
= & \stackrel[0]{T}{\int}\left(\S_{\Pi_{n}}f_{0}\left(p_{\Pi_{n}}\left(t\right)\right),\S_{\Pi_{n}}q_{\Pi_{n}}\right)_{\Ls^{2}\left(\Omega\right)}\varphi\left(t\right)\dd t\nonumber \\
- & b\beta\xi\stackrel[0]{T}{\int}\left(\S_{\Pi_{n}}\varLambda\left(g\left(u_{\Pi_{n}}\left(t\right),p_{\Pi_{n}}\left(t\right)\right)\right),\S_{\Pi_{n}}q_{\Pi_{n}}\right)_{\Ls^{2}\left(\Omega\right)}\varphi\left(t\right)\dd t.\nonumber 
\end{align}
By applying integration by parts to the first term of the equality
and using the properties of $\varphi$, we get 

\begin{align}
- & \overbrace{\alpha\xi^{2}\stackrel[0]{T}{\int}\left(\S_{\Pi_{n}}p_{\Pi_{n}}\left(t\right),\S_{\Pi_{n}}q_{\Pi_{n}}\right)_{\Ls^{2}\left(\Omega\right)}\dot{\varphi}\left(t\right)\dd t}^{\text{I.}}\label{eq:beforeLimFinalEq}\\
+ & \overbrace{\xi^{2}\stackrel[0]{T}{\int}\underset{K\in\Pi_{n}}{\sum}\underset{\sigma\in\E_{K}}{\sum}F_{K,\sigma}\left(t\right)q_{K}\varphi\left(t\right)\dd t}^{\text{II.}}\nonumber \\
= & \overbrace{\stackrel[0]{T}{\int}\left(\S_{\Pi_{n}}f_{0}\left(p_{\Pi_{n}}\left(t\right)\right),\S_{\Pi_{n}}q_{\Pi_{n}}\right)_{\Ls^{2}\left(\Omega\right)}\varphi\left(t\right)\dd t}^{\text{III.}}\nonumber \\
- & \overbrace{b\beta\xi\stackrel[0]{T}{\int}\left(\S_{\Pi_{n}}\varLambda\left(g\left(u_{\Pi_{n}}\left(t\right),p_{\Pi_{n}}\left(t\right)\right)\right),\S_{\Pi_{n}}q_{\Pi_{n}}\right)_{\Ls^{2}\left(\Omega\right)}\varphi\left(t\right)\dd t}^{\text{IV.}}.\nonumber 
\end{align}
We investigate the limits of the individual terms in (\ref{eq:beforeLimFinalEq}),
considering again a suitable mesh subsequence $\Pi_{k_{n}}$ (see
Lemma \ref{lem:convergence-to-p}) denoted as $\Pi_{n}$ for brevity.

Thanks to the relationships (\ref{eq:interpolationToOrg}) and (\ref{eq:conv2}),
taking the limit of I. results in

\begin{align*}
\alpha\xi^{2}\stackrel[0]{T}{\int}\left(\S_{\Pi_{n}}p_{\Pi_{n}}\left(t\right),\S_{\Pi_{n}}q_{\Pi_{n}}\right)_{\Ls^{2}\left(\Omega\right)}\dot{\varphi}\left(t\right)\dd t & \rightarrow\alpha\xi^{2}\stackrel[0]{T}{\int}\left(p,q\right)_{\Ls^{2}\left(\Omega\right)}\dot{\varphi}\left(t\right)\dd t.
\end{align*}
Since the strong convergence $\S_{\Pi_{n}}p_{\Pi_{n}}\rightarrow p$
in $\Ls^{2}\left(\J;\Ls^{2}\left(\Omega\right)\right)$ signifies
convergence almost everywhere in $\J\times\Omega$ \cite{Kufner-Function-Spaces}
the limit of III. may be taken

\[
\stackrel[0]{T}{\int}\left(\S_{\Pi_{n}}f_{0}\left(p_{\Pi_{n}}\left(t\right)\right),\S_{\Pi_{n}}q_{\Pi_{n}}\right)_{\Ls^{2}\left(\Omega\right)}\varphi\left(t\right)\dd t\rightarrow\stackrel[0]{T}{\int}\left(f_{0}\left(p\right),q\right)_{\Ls^{2}\left(\Omega\right)}\varphi\left(t\right)\dd t.
\]
Similarly using (\ref{eq:interpolationToOrg}) the limit of IV. may
be taken

\begin{align*}
b\beta\xi\stackrel[0]{T}{\int}\left(\S_{\Pi_{n}}\varLambda\left(g\left(u_{\Pi_{n}}\left(t\right),p_{\Pi_{n}}\left(t\right)\right)\right),\S_{\Pi_{n}}q_{\Pi_{n}}\right)_{\Ls^{2}\left(\Omega\right)}\varphi\left(t\right)\dd t\\
\rightarrow b\beta\xi F\stackrel[0]{T}{\int}\left(\varLambda\left(g\left(u,p\right)\right),q\right)_{\Ls^{2}\left(\Omega\right)}\varphi\left(t\right)\dd t.
\end{align*}
Since we are considering the homogeneous Dirichlet boundary condition
for $p$, the term V. is equal to zero.

Using Lemma \ref{lem:gradient-convergence}, the limit of II. reads

\[
\xi^{2}\stackrel[0]{T}{\int}\left(\underset{K\in\Pi_{n}}{\sum}\underset{\sigma\in\E_{K}}{\sum}F_{K,\sigma}\left(t\right)q_{K}\left(\vec{x}\right)\right)\varphi\left(t\right)\dd t\rightarrow\xi^{2}\stackrel[0]{T}{\int}\left(\underset{\Omega}{\int}\nabla p\left(t,\vec{x}\right)\cdot\nabla q\left(\vec{x}\right)\dd\vec{x}\right)\varphi\left(t\right)\dd t.
\]

After passing to the limit, the relationship (\ref{eq:beforeLimFinalEq})
becomes the weak equality (\ref{eq:weakFormulation-p}), i.e. $p$
is the solution of the phase field equation. A similar procedure may
be performed to show that $u$ is the weak solution of the heat equation.
The uniqueness of the solution may be shown using a similar procedure
as in \cite{Benes-Diffuse-Interface}, using the specific form of
$f(u,p,\xi)=f_{0}\left(p\right)-b\beta\xi\varLambda\left(g\left(u,p\right)\right)$.%
\begin{lyxgreyedout}
\textbf{NOTE (Ales):} Actually it can be proved in a more elegant
way using Green's functions and extending local existence to global
using invariant regions, maybe we can also mention this. Aleš comment:
Let's not discuss this in this article, it requires a lot of definitions
and would be confusing at this point.%
\end{lyxgreyedout}
{} This implies that all convergent subsequences $\left(\S_{\Pi_{k_{n}}}u_{\Pi_{k_{n}}},\S_{\Pi_{k_{n}}}p_{\Pi_{k_{m}}}\right)$
have the same unique limit and thus the whole sequence $\left(\S_{\Pi_{n}}u_{\Pi_{n}},\S_{\Pi_{n}}p_{\Pi n}\right)$
converges to $\left(u,p\right)$ in $\Ls^{2}\left(\J;\Ls^{2}\left(\Omega\right)\right).$%
\begin{lyxgreyedout}
\textbf{NOTE:} We show this in terms of (an almost) general sequence
in a metric space $\left(M,\rho\right)$.

Assumptions:
\begin{enumerate}
\item There exists a subsequence $p_{k_{n}}$ such that $p_{k_{n}}\to p$
\item For each such possible convergent subsequence $p_{k_{n}'}$, the limit
is unique, i.e. $p_{k_{n}'}\to p$
\end{enumerate}
Consequence: $p_{n}\to p$. Assume the opposite, i.e.
\[
p_{n}\nrightarrow p\;\iff\;\left(\exists\varepsilon>0\right)\left(\forall n_{0}\right)\left(\exists n>n_{0}\right)\left(\rho\left(p_{n},p\right)>\varepsilon\right).
\]
This yields the existence of a subsequence $p_{l_{n}}$ satisfying
$\left(\forall n\in\N\right)\left(\rho\left(p_{l_{n}},p\right)>\varepsilon\right)$.
However, this sequence has the same properties (beyond the structure
of the metric space) as $p_{n}$ and thus it contains a convergent
subsequence $p_{l_{k_{n}}}$. As $p_{l_{k_{n}}}$ is also a convergent
subsequence of $p_{n}$, its limit is $p$, which is a contradiction
with $\left(\forall n\in\N\right)\left(\rho\left(p_{l_{k_{n}}},p\right)>\varepsilon\right)$.%
\end{lyxgreyedout}
\end{proof}

\section{Conclusion}

This paper provides a detailed proof of existence of the weak solution
and convergence of the finite volume scheme to this solution for the
isotropic phase field model suitable for solidification modeling in
polyhedral domains covered by admissible polyhedral meshes. We consider
a general form of the reaction term in the phase field equation which
allows to apply the presented results to existing models \cite{Kobayashi-PF-Dendritic}
as well as several new variants of the phase field model presented
in our work \cite{arXiv-PhaseField-Focusing-Latent-Heat}. We show
that introducing an artificial limiter of the reaction term makes
it possible to perform the analysis while not affecting the simulation
results \cite{arXiv-PhaseField-Focusing-Latent-Heat}. A semi-discrete
form of the scheme is used, leaving temporal discretization up to
the reader's choice.

%LATEX_EXPORT_END

\bigskip{}
\noindent \textbf{\small{}Acknowledgment:}\\
{\small{}
This work is part of the project \emph{Centre of Advanced Applied
Sciences} (Reg. No. CZ.02.1.01/0.0/0.0/16-019/0000778), co-financed
by the European Union. Partial support of grant No. SGS20/184/OHK4/3T/14
of the Grant Agency of the Czech Technical University in Prague.
}

\bibliographystyle{model1b-num-names}
%\bibliography{References/publications,References/references_MR-DTI,References/references_MATH-PHYS,References/references_IT}
\bibliography{References/publications,References/references_MATH-PHYS}

\end{document}